\documentclass{amsart}
\usepackage{latexsym,amssymb}
\usepackage[english]{babel}
\usepackage{bm}
\usepackage{amsfonts, amsthm}
\usepackage{amsmath}
\usepackage{mathrsfs}

\usepackage{geometry}
\newtheorem{prop}{Proposition}[section]
\newtheorem{lem}[prop]{Lemma}
\newtheorem{cor}[prop]{Corollary}
\newtheorem{thm}[prop]{Theorem}
\newtheorem{conj}[prop]{Conjecture}
\theoremstyle{definition}
\newtheorem{rem}[prop]{Remark}
\newtheorem{defi}[prop]{Definition}
\newtheorem{ex}[prop]{Example}

\numberwithin{equation}{section}

\def\G{\Gamma}

\def\H{\mathrm{H}}

\def\Z{\mathbb{Z}}

\def\E{\mathcal{E}}
\def\D{\mathcal{D}}
\def\B{\mathcal{B}}
\def\F{\mathcal{F}}

\begin{document}

\title[A generalization of Heffter arrays]{A generalization of Heffter arrays}

\author[S. Costa]{Simone Costa}
\address{DII/DICATAM - Sez. Matematica, Universit\`a degli Studi di Brescia, Via
Branze 38, I-25123 Brescia, Italy}
\email{simone.costa@unibs.it}

\author[F. Morini]{Fiorenza Morini}
\address{Dipartimento di Scienze Matematiche, Fisiche e Informatiche, Universit\`a di Parma,
Parco Area delle Scienze 53/A, I-43124 Parma, Italy}
\email{fiorenza.morini@unipr.it}

\author[A. Pasotti]{Anita Pasotti}
\address{DICATAM - Sez. Matematica, Universit\`a degli Studi di Brescia, Via
Branze 43, I-25123 Brescia, Italy}
\email{anita.pasotti@unibs.it}

\author[M.A. Pellegrini]{Marco Antonio Pellegrini}
\address{Dipartimento di Matematica e Fisica, Universit\`a Cattolica del Sacro Cuore, Via Musei 41,
I-25121 Brescia, Italy}
\email{marcoantonio.pellegrini@unicatt.it}

\begin{abstract}
In this paper we define a new class of partially filled arrays, called relative Heffter arrays, that are a
generalization of the Heffter arrays introduced by Archdeacon in 2015.
Let $v=2nk+t$ be a positive integer, where $t$ divides $2nk$, and
let $J$ be the subgroup of $\Z_v$ of order $t$.  A $\H_t(m,n; s,k)$ Heffter array  over $\Z_v$ relative to $J$
  is an $m\times n$ partially filled  array
 with elements in $\Z_v$ such that:
(a) each row contains $s$ filled cells and each column contains $k$ filled cells;
(b) for every $x\in \Z_v\setminus J$, either $x$ or $-x$ appears in the array;
(c) the elements in every row and column sum to $0$.
Here we study the existence of square integer
(i.e. with entries chosen in $\pm\left\{1,\dots,\left\lfloor \frac{2nk+t}{2}\right\rfloor \right\}$ and where the sums are zero in $\mathbb{Z}$)
relative Heffter arrays for $t=k$, denoted by $\H_k(n;k)$. In particular, we prove that
for $3\leq k\leq n$, with $k\neq 5$,
  there exists an integer $\H_k(n;k)$ if and only if one of the following holds:
(a) $k$ is odd and $n\equiv 0,3\pmod 4$;
(b) $k\equiv 2\pmod 4$ and $n$ is even;
(c) $k\equiv 0\pmod 4$.
Also, we show how these arrays give rise to cyclic cycle decompositions of the complete multipartite graph.
\end{abstract}


\maketitle

\section{Introduction}\label{sec:Intro}
An $m \times n$  partially filled (p.f., for short) array on a set $\Omega$ is an $m \times n$ matrix
whose elements belong to $\Omega$ and where we also allow some cells to be empty.
An interesting class of p.f. arrays, called Heffter arrays, has been introduced by Dan Archdeacon in \cite{A}.

\begin{defi}\label{def:H}
A \emph{Heffter array} $\H(m,n; s,k)$ is an $m \times n$ p.f. array with elements in $\Z_{2nk+1}$ such that
\begin{itemize}
\item[(\rm{a})] each row contains $s$ filled cells and each column contains $k$ filled cells;
\item[(\rm{b})] for every $x\in \Z_{2nk+1}\setminus\{0\}$, either $x$ or $-x$ appears in the array;
\item[(\rm{c})] the elements in every row and column sum to $0$ (in $\Z_{2nk+1}$).
\end{itemize}
\end{defi}

Trivial necessary conditions for the existence of an $\H(m,n; s,k)$ are $ms=nk$, $3\leq s \leq n$ and $3\leq k \leq m$.
Hence if the Heffter array is square, namely if $m=n$, then $s=k$; such an array will be denoted by $\H(n;k)$.
A Heffter array is called \emph{integer} if Condition (c) in Definition \ref{def:H} is strengthened so that the
elements in every row and in every
column, viewed as integers in $\{\pm 1, \ldots, \pm nk\}$, sum to zero in $\Z$.

Heffter arrays are considered interesting and worthy of study in their own right together with their vast variety of applications.
In fact, there are some recent papers in which they are investigated since they allow to obtain new biembeddings
(see \cite{A,CDDYbiem, CMPPHeffter, DM}), while other ones completely solve the existence problem of square Heffter arrays
(see \cite{ABD, ADDY,  BCDY, CDDY,  DW}). In particular, in \cite{ADDY,DW} the authors verify the existence of a square integer Heffter array for all admissible orders,
proving the following theorem.
\begin{thm}
  There exists an integer $\H(n;k)$ if and only if $3\leq k\leq n$ and $nk\equiv 0,3\pmod 4$.
\end{thm}

In this paper we introduce a new class of p.f. arrays, which is a natural generalization of Heffter
arrays.

\begin{defi}\label{def:RelativeH}
Let $v=2nk+t$ be a positive integer, where $t$ divides $2nk$, and let $J$ be the subgroup of $\Z_v$ of order $t$.
 A $\H_t(m,n; s,k)$ \emph{Heffter array  over $\Z_v$ relative to $J$} is an $m\times n$ p.f.  array
 with elements in $\Z_v$ such that:
\begin{itemize}
\item[($\rm{a_1})$] each row contains $s$ filled cells and each column contains $k$ filled cells;
\item[($\rm{b_1})$] for every $x\in \Z_{v}\setminus J$, either $x$ or $-x$ appears in the array;
\item[($\rm{c_1})$] the elements in every row and column sum to $0$ (in $\Z_v$).
\end{itemize}
\end{defi}
If $\H_t(m,n; s,k)$ is a square array, it will be denoted by $\H_t(n;k)$.
Clearly, if $t=1$, namely if $J$ is the trivial subgroup of $\Z_{2nk+1}$, we find again the classical concept of Heffter array.
A relative Heffter array is called \emph{integer} if Condition ($\rm{c_1}$) in Definition \ref{def:RelativeH} is
strengthened so that the elements in every row and in every
column, viewed as integers in
$\pm\left\{ 1, \ldots, \left\lfloor \frac{2nk+t}{2}\right\rfloor \right\} ,$
sum to zero in $\Z$.
The \emph{support} of an integer Heffter array $A$, denoted by $supp(A)$, is defined to be the set of the absolute values of the elements
contained in $A$.
It is immediate to see that an integer $\H_2(n;k)$ is nothing but an integer $\H(n;k)$, since in both cases the support is $\{1,2,\ldots,nk\}$.

\begin{ex}\label{ex:43}
The following are integer  relative Heffter arrays $\H_{16}(4;4)$ and  $\H_{32}(4;$ $4)$, respectively.
\begin{footnotesize}
 $$\begin{array}{|c|c|c|c|}\hline
   1&   -7&  -16  &  22   \ \\ \hline
   23 &   2 &   -8 &  -17  \\ \hline
  -13 &    19 &   4 &    -10 \\ \hline
  -11 &   -14 &    20 &  5 \\ \hline
  \end{array}
  \hspace{2cm}
  \begin{array}{|c|c|c|c|}\hline
   1&   -9&  -21  &  29   \ \\ \hline
   31 &   3 &   -11 &  -23  \\ \hline
  -17 &    25 &   5 &    -13 \\ \hline
  -15 &   -19 &    27 &  7 \\ \hline
  \end{array}$$
\end{footnotesize}
\end{ex}

Here we investigate the existence problem of this new class of arrays in the square integer case.
In Section \ref{sec:relativeDF} we will describe the relationship between relative Heffter arrays and \emph{relative difference families}, see \cite{AB},
which are very useful tools to obtain regular graph decompositions.
In fact, many known results about regular decompositions of the complete graph and of the complete multipartite graph have been obtained thanks to difference families and to relative difference families, respectively.
From this relationship it follows that starting from a relative Heffter array it is possible to construct a pair of orthogonal cyclic cycle decompositions of the complete
multipartite graph, as we will explain in details in the same section.
In Section \ref{sec:necessary} we will determine some necessary conditions for the existence of an
integer relative Heffter array $\H_t(n;k)$.
In Section \ref{sec:extension} we will present a result which reduces the existence problem of an integer $\H_k(n;k)$
to the case $3\leq k\leq 6$,
then in Section \ref{sec:constructions} we will present direct constructions for these basic cases.
The results of these two sections allow us to present  an almost complete result which can be summarized as follows.
\begin{thm}\label{thm:esistenza}
Let $3\leq k\leq n$ with $k\neq 5$.
  There exists an integer $\H_k(n;k)$ if and only if one of the following holds:
  \begin{itemize}
    \item $k$ is odd and $n\equiv 0,3\pmod 4$;
    \item $k\equiv 2\pmod 4$ and $n$ is even;
    \item $k\equiv 0\pmod 4$.
  \end{itemize}
Furthermore, there exists an integer $\H_5(n;5)$ if
$n\equiv 3\pmod 4$ and it does not exist if $n\equiv 1,2\pmod 4$.
\end{thm}
Note that  for $k=5$ we solved the existence problem of integer relative Heffter arrays $\H_5(n;5)$ only for
$n\equiv 3\pmod 4$, leaving the case $n\equiv 0 \pmod 4$ open.
In Section \ref{sec:conclusion} we prove Theorem \ref{thm:esistenza} and the result about orthogonal decompositions obtained thanks to the arrays constructed in previous sections.
Hence, in this paper we focus on the construction of relative Heffter arrays $\H_k(n;k)$. Further constructions for integer $\H_t(n;k)$ will be given in \cite{MP}.
We have to point out that relative Heffter arrays, as well as the classical ones, are useful to obtain biembeddings of orthogonal cyclic cycle decompositions.
This relationship is investigated in \cite{CMPPBiembeddings}.

\section{Relation with relative difference families and decompositions of the complete multipartite graph}\label{sec:relativeDF}

Firstly, we recall some basic definitions about graphs and graph decompositions.
Given a graph $\G$, by $V(\G)$ and $E(\G)$ we mean the vertex set and the edge set of $\G$, respectively,
and by $\lambda \G$  the multigraph obtained from $\G$ by repeating each edge $\lambda$ times.
We will denote by $K_v$ the complete graph of order $v$ and by
 $K_{q\times r}$ the complete multipartite graph with $q$ parts, each of size $r$.
Obviously  $K_{q\times 1}$ is nothing but the complete graph $K_q$.
The cycle of length $k$, also called $k$-cycle, will be denoted by $C_k$.

The following are well known definitions and results which can be found in \cite{B}.
Let $\G$ be a subgraph of a graph $K$.
A $\Gamma$-\emph{decomposition} of $K$ is a set $\D$
of subgraphs of $K$ isomorphic to $\G$ whose edges partition $E(K)$.
If the vertices of $\G$ belong to an additive group $G$, given $g\in G$,  the graph whose vertex set is
$V(\G)+g$ and whose edge set is $\{\{x+g,y+g\}\mid \{x,y\}\in E(\G)\}$ will be denoted by $\G+g$.
An \emph{automorphism group} of a $\G$-decomposition $\D$ of $K$ is a group of bijections on $V(K)$
leaving $\D$ invariant.
 A $\G$-decomposition $\D$ of $K$ is said to be \emph{regular under a group} $G$ or $G$-\emph{regular}
 if it admits $G$ as an automorphism group acting sharply transitively on $V(K)$.
 Here we consider cyclic cycle decompositions, namely $\G$-decompositions which are regular
 under a cyclic group where $\G$ is a cycle.

\begin{prop}
Given an additive group $G$,  a $\G$-decomposition $\D$ of a graph $K$ is $G$-regular if and only if, up to isomorphisms, the following conditions hold:
  \begin{itemize}
    \item $V(K)=G$;
    \item $B\in \D$ $\Rightarrow$ $B+g\in \D$ for all $g\in G$.
  \end{itemize}
\end{prop}

One of the most efficient tools applied for finding regular decompositions is the \emph{difference method}.
Here, in particular, we are interested in relative difference families over graphs, introduced in \cite{BP} (see also \cite{BP2}).

\begin{defi}
Let  $\G$ be a graph with vertices in an additive group $G$.
The multiset
$$\Delta \G = \{\pm(x-y)\mid  \{x,y\}\in E(\G)\}$$
is called the \emph{list of differences} from $\G$.
\end{defi}

More generally, given a set $\mathcal{W}$ of graphs with vertices
in $G$, by $\Delta \mathcal{W}$ one means the union (counting
multiplicities) of all multisets $\Delta \G$,
where  $\G\in \mathcal{W}$.

\begin{defi}\label{def:DF}
  Let $J$ be a subgroup of an additive group $G$ and let $\G$ be a graph.
  A collection $\F$ of graphs isomorphic to $\G$ and with vertices in $G$
  is said to be a $(G,J,\G,\lambda)$-\emph{difference family} (briefly, DF)
if each element of $G\setminus J$ appears exactly $\lambda$ times in $\Delta \F$,
while no element of $J$ appears there.
\end{defi}
One speaks also of a \emph{difference family over $G$ relative to $J$}.
If $J=\{0\}$ one simply says that $\F$ is a $(G,\G,\lambda)$-DF.
If $\G$ is a complete graph we find again the concept introduced by Buratti in \cite{BDF}.
If $J=\{0\}$ and $\G$ is a complete graph, then we obtain the classical concept of difference family, see \cite{AB}.
If $t$ is a divisor of $v$, by writing $(v,t,\G,\lambda)$-DF one means a $(\Z_v,\frac{v}{t}\Z_v,\G,\lambda)$-DF,
where $\frac{v}{t}\Z_v$ denotes the subgroup of $\Z_v$ of order $t$.
We point out that the most interesting (and the most difficult) case is with $\lambda=1$.
The relationship between relative difference families and regular decompositions of the complete multipartite graph is explained in the following result.

\begin{thm}\label{thm:basecycles}\cite[Proposition 2.6]{BP}
If $\F=\{B_1,\ldots,B_\ell\}$ is a $(G,J,\G,\lambda)$-DF, then the collection of graphs $\B=\{B_i+g \mid i=1,\ldots,\ell; g\in G\}$
is a $G$-regular $\G$-decomposition of $\lambda K_{q\times r}$, where $q=|G:J|$ and $r=|J|$.
Thus, in particular, a $(G,\G,\lambda)$-DF gives rise to a $G$-regular $\G$-decomposition of $\lambda K_{|G|}$.
\end{thm}
Results about regular cycle decompositions of the complete multipartite graph via relative difference families can be found in \cite{BP,BMT,MPP,PP}.

Now, in order to present the connection between relative Heffter arrays and relative difference families,
we have to introduce the concept of \emph{simple ordering}.

Henceforward, given two integers $a\leq b$, we denote by $[a,b]$ the interval containing the integers $\{a,a+1,\ldots,b\}$.
If $a>b$, then $[a,b]$ is empty.
Given an $m\times n$ p.f.  array $A$,
the rows and the columns of $A$ will be denoted by $\overline{R}_1,\ldots,\overline{R}_m$ and by $\overline{C}_1,\ldots,\overline{C}_n$, respectively.
We will denote by $\E(A)$ the list of the elements of the filled cells of $A$.
Analogously, by $\E(\overline{R}_i)$ and $\E(\overline{C}_j)$ we mean the elements  of the $i$-th row and of the $j$-th column of $A$, respectively.
Given  a finite subset $T$ of an abelian group $G$ and an ordering $\omega=(t_1,t_2,\ldots,t_k)$ of the elements
in $T$,  let $s_i=\sum_{j=1}^i t_j$, for any $i\in[1,k]$,
be the $i$-th partial sum of $T$.
The ordering $\omega$ is said to be \emph{simple} if $s_b\neq s_c$ for all $1\leq b <  c\leq k$. 
If $s_k=0$, this is equivalent to require that no proper subsequence of consecutive elements of $\omega$ sums to $0$. Note that if $\omega$ is a simple ordering then also $\omega^{-1}=(t_k,t_{k-1},\ldots,t_1)$ is simple.
We point out that there are several interesting problems and conjectures about distinct partial sums: see, for instance, \cite{AL, ADMS, CMPPSums, HOS, O}.
Given an $m \times n$ p.f. array $A$, by $\omega_{\overline{R}_i}$ and $\omega_{\overline{C}_j}$ we will
denote  an ordering of $\E(\overline{R}_i)$ and of $\E(\overline{C}_j)$, respectively.  If for any $i\in[1, m]$
and for any $j\in[1,n]$, the orderings $\omega_{\overline{R}_i}$ and $\omega_{\overline{C}_j}$ are simple, we define
by
  $\omega_r=\omega_{\overline{R}_1}\circ \ldots \circ\omega_{\overline{R}_m}$ the simple ordering for the rows and
  by $\omega_c=\omega_{\overline{C}_1}\circ \ldots \circ\omega_{\overline{C}_n}$ the simple ordering for the columns.
  A p.f. array $A$ on a group $G$ is said to be
 \emph{simple} if there exists a simple ordering for each row and each column of $A$. Clearly if $k\leq 5$, then every relative Heffter array is simple.
Note that if we have a simple $\H_t(n;k)$ we can construct $2^n$ simple orderings $\omega_r$ for the rows
and $2^n$ simple orderings $\omega_c$ for the columns, since  the inverse of a simple ordering of a row (or a column) is still a simple ordering.

\begin{prop}\label{from Heffter to DF}
If $A$ is a simple $\H_t(m,n;s,k)$, then
there exist a $(2ms+t,t,C_s,$ $1)$-DF
and  a $(2nk+t,t,C_k,1)$-DF.
\end{prop}

\begin{proof}
By hypothesis $A$ is simple, hence there exists a simple ordering $\omega_i$ for
the $i$-th
row of $A$ with $i\in [1,m]$.
So, from each row of $A$ we can construct an $s$-cycle whose vertices in $\Z_{2ms+t}$ are the partial
sums of $\omega_i$.
Let $\F_s$ be the set of $m$ $s$-cycles constructed this way using the rows of $A$.
Clearly, $\Delta \F_s=\pm \E(A)$.
On the other hand, since $A$ is a $\H_t(m,n;s,k)$ we have $\pm \E(A)=\Z_{2ms+t}\setminus \frac{2ms+t}{t}\Z_{2ms+t}$.
Hence $\F_s$ is a $(2ms+t,t,C_s,1)$-DF.

An analogous reasoning can be done on the columns of $A$ obtaining a $(2nk+t,t,C_k,1)$-DF, say $\F_k$ similarly.
\end{proof}

\begin{rem}\label{rem:ortho}
Let $\F_s$ and $\F_k$ be the relative difference families constructed in the previous proposition.
  Note that for any $C_s\in \F_s$ and any $C_k\in \F_k$, we have
$|\Delta C_s \cap \Delta C_k| \in \{0,2\}.$
\end{rem}

\begin{ex}
Starting from the array $A=\H_{16}(4;4)$ given in Example \ref{ex:43} we construct two $(48,16,C_4,1)$-DFs.
Since $k=4<6$  every ordering is simple.
Set, for instance:
\begin{align*}
\omega_1&=(1,-7,-16,22), & \nu_1&=(-11,-13,23,1), \\
\omega_2&=(23,2,-8,-17), & \nu_2&=(-7,2,19,-14), \\
 \omega_3&=(-10,4,19,-13), &  \nu_3&=(20,4,-8,-16), \\
  \omega_4&=(-11,-14,20,5),& \nu_4&=(5,-10,-17,22).
\end{align*}
The $\omega_i$'s and the $\nu_i$'s are simple orderings for the rows and the columns of $A$, respectively.
Starting from these orderings we obtain the following $4$-cycles:
\begin{align*}
C^{\omega_1}&=(1,-6,-22,0), & C^{\nu_1}&=(-11,-24,-1,0),\\
 C^{\omega_2}&=(23,25,17,0), & C^{\nu_2}&=(-7,-5,14,0), \\
 C^{\omega_3}&=(-10,-6,13,0), & C^{\nu_3}&=(20,24,16,0), \\
 C^{\omega_4}&=(-11,-25,-5,0), & C^{\nu_4}&=(5,-5,-22,0).
\end{align*}
Set $\F^\omega_4=\{C^{\omega_i}\mid i\in[1,4]\}$ and $\F^\nu_4=\{C^{\nu_i}\mid i\in[1,4]\}$;
by the construction of the cycles it immediately follows that
$\Delta \F^\omega_4=\Delta \F^\nu_4= \Z_{48}\setminus 3\Z_{48}$.
Hence $\F^\omega_4$ and $\F^\nu_4$ are two $(48,16,C_4,1)$-DFs.
\end{ex}

We recall the following definition, see for instance \cite{CY1}.
\begin{defi}
  Two $\G$-decompositions $\D$ and $\D'$ of a simple graph $K$ are said to be \emph{orthogonal} if  for any $B$ of $\D$
  and any $B'$ of $\D'$, $B$ intersects $B'$ in at most one edge.
\end{defi}

\begin{prop}\label{HeffterToDecompositions}
  Let  $\H_t(m,n;s,k)$ be simple with respect to the orderings $\omega_r$ and $\omega_c$. Then:
  \begin{itemize}
    \item[(1)] there exists a cyclic $s$-cycle decomposition $\D_{\omega_r}$ of $K_{\frac{2ms+t}{t}\times t}$;
    \item[(2)] there exists a cyclic $k$-cycle decomposition $\D_{\omega_c}$ of $K_{\frac{2nk+t}{t}\times t}$;
    \item[(3)] the cycle decompositions $\D_{\omega_r}$ and $\D_{\omega_c}$ are orthogonal.
  \end{itemize}
\end{prop}

\begin{proof}
(1) and (2) follow from Theorem \ref{thm:basecycles} and Proposition \ref{from Heffter to DF}.
Then (3) follows from Remark \ref{rem:ortho}.
\end{proof}

\section{Necessary conditions for the existence of square integer $\H_t(n;k)$}\label{sec:necessary}

Here we determine some necessary conditions for the existence of square integer $\H_t(n;k)$.
We recall that, by definition, $t$ divides $2nk$.

\begin{prop}\label{prop:necc}
Suppose that there exists an integer $\H_t(n;k)$.
\begin{itemize}
\item[(1)] If $t$ divides $nk$, then
$$nk\equiv 0 \pmod 4 \quad \textrm{ or } \quad nk\equiv -t \equiv \pm 1\pmod 4.$$
\item[(2)] If $t=2nk$, then $k$ must be even.
\item[(3)] If $t\neq 2nk$ does not divide $nk$, then
$$t+2nk\equiv 0 \pmod 8.$$
\end{itemize}
\end{prop}

\begin{proof}
Given an integer $\H_t(n;k)$, in order for each row to sum to zero, each row must contain an even number of odd
numbers. In particular, the entire array contains an even number of odd numbers.
The support of $\H_t(n;k)$ is the set $S=\left[1,nk+\left\lfloor\frac{t}{2}\right\rfloor\right]\setminus T$,
where $T$ consists of the multiples of $\frac{2nk+t}{t}$, i.e.,
$$T=\left\{\frac{2nk+t}{t},2\frac{2nk+t}{t},\ldots, \left\lfloor\frac{t}{2}\right\rfloor\frac{2nk+t}{t}
\right\}.$$
Note that the interval $\left[1,nk+\left\lfloor\frac{t}{2}\right\rfloor\right]$ contains exactly
$\left\lfloor\frac{nk+\left\lfloor\frac{t}{2}\right\rfloor+1}{2}\right\rfloor$
odd numbers. Now, if $\frac{2nk+t}{t}$ is odd (i.e., if $t$ divides $nk$), then $T$ contains
$\left\lfloor\frac{\left\lfloor\frac{t}{2}\right\rfloor+1}{2}\right\rfloor$ odd numbers.
It follows that
$$\left\lfloor\frac{nk+\left\lfloor\frac{t}{2}\right\rfloor+1}{2}\right\rfloor-\left\lfloor\frac{\left\lfloor\frac{t}{2}
\right\rfloor+1}{2}\right\rfloor$$
is necessarily even, giving case (1).

 If $\frac{2nk+t}{t}$ is even (i.e., if $t$ does not divide $nk$), then $T$ contains no odd numbers.
Hence $\left\lfloor\frac{nk+\left\lfloor\frac{t}{2}\right\rfloor+1}{2}\right\rfloor$ must be even.
In particular, if
$t=2nk$, then $T$ contains all the even numbers of $[1,2nk]$, and so $S$ consists only of odd numbers.
It follows that $k$ must be even, giving case (2).
We are left to consider case (3). If $t\neq 2nk$ does not divide $nk$, then $t$ must be even and
$$\left\lfloor\frac{2nk+t+2}{4}\right\rfloor$$
is necessarily even. Hence $2nk+t \equiv 0,6\pmod8$. Now we will show that
$2nk+t \equiv 6\pmod8$ leads to a contradiction.
Since we are in the hypothesis that  $t$ is even, we can set $t=2m$ with $nk=mh$.
From $2nk+t \equiv 6\pmod8$ we obtain
$2mh+2m \equiv 6\pmod 8$ which implies $ m(h+1)\equiv3 \pmod4$.
In particular this implies that  $h$ is even which contradicts the hypothesis that $t=2m$
does not divides $nk=mh$.
Hence (3) follows.
\end{proof}

We have to point out that the necessary conditions  of the previous proposition are not sufficient. In fact, for $k=3$ we have also found
 two non-existence results. In order to present them we need some definitions.

Given an $m\times n$ p.f. array $A$,  by $A[i,j]$ we mean the element of $A$ in position $(i,j)$.
Also, we define the skeleton of $A$, denoted by $skel(A)$, to be the set of the filled positions of $A$.
In case $A$ and $B$ are $m\times n$ p.f. arrays such that $skel(A)\cap skel(B)=\emptyset$,
we define the union of $A$ and $B$ to be the $m\times n$ p.f. array filled with both the entries of $A$ and $B$.

Let $A$ be an $m \times n$ p.f. array with no empty rows and no empty columns.
Let $\mathcal{R}$ be an $r\times n$ subarray of $A$ and $\mathcal{C}$ be an $m\times c$ subarray of $A$.
We say that the $r \times c$ subarray $\mathcal{R}\cap \mathcal{C}$ of $A$ is \emph{closed}  if  $skel(\mathcal{R}\cap\mathcal{C})=skel(\mathcal{R})\cup skel(\mathcal{C})$.
We say that a closed subarray is minimal if it is
minimal with respect to the inclusion.

\begin{ex}
Consider the following p.f. array $A$, where a filled cell is represented by a $\bullet$:
 \begin{center}
\begin{footnotesize}
$\begin{array}{|r|r|r|r|r|r|}\hline
\bullet & \bullet &   &   &   & \\\hline
  &     & \bullet &   &\bullet &  \\ \hline
	\bullet & \bullet & & \bullet & & \\\hline
	& & \bullet & & \bullet & \bullet\\\hline
	\bullet & \bullet & & \bullet & & \\\hline
	&  & \bullet &  &  & \bullet\\\hline
\end{array}$
\end{footnotesize}
\end{center}

\noindent Let $\mathcal{R}$ be the subarray consisting of the rows $1,3,5$ of $A$ and
let $\mathcal{C}$ be the subarray consisting of the columns $1,2,4$.
Then $\mathcal{R}\cap \mathcal{C}$ is a minimal closed subarray of $A$.
\end{ex}

\begin{lem}\label{3n}
There is no integer $\H_{3n}(n;3)$ for $n\geq 3$.
\end{lem}

\begin{proof}
By contradiction, suppose that $A$ is an integer $\H_{3n}(n;3)$, hence
by Proposition \ref{prop:necc} we have $n \equiv 0\pmod 4$.
Then, $supp(A)=\left[1,\frac{9n}{2}\right]\setminus \left\{3,6,\ldots,\frac{9n}{2}\right\}$.
Fix any row of $A$ and consider its three elements. Since they sum to zero in $\Z$ and each of them cannot be congruent to zero modulo $3$,
these elements must belong to the same residue class modulo $3$.
The same clearly holds also for any column of $A$. So, considering alternatively rows and columns, one obtains that, for any minimal closed subarray $B$ of $A$,
all the elements of $\E(B)$ belong to the same residue class modulo $3$.
It is easy to see that if we change all the signs of the elements of a closed subarray of $A$, we still obtain an integer $\H_{3n}(n;3)$.
Hence, there would exist an integer $\H_{3n}(n;3)$, say $A'$, such that all its elements
 belong to the same residue class modulo $3$ and, without loss of generality, we can suppose that this residue class modulo $3$ is $1$, namely that
$$\E(A')=\left\{1,-2, 4,-5,7,-8,\ldots,\frac{9n-4}{2}, -\frac{9n-2}{2}\right\}.$$
Now, it is evident that the elements of $\E(A')$ cannot sum to zero in $\Z$, giving a contradiction.
\end{proof}

\begin{lem}
There is no integer $\H_8(4;3)$.
\end{lem}

\begin{proof}
Assume, by way of a contradiction, that $A$ is an integer $\H_8(4;3)$, hence $supp(A)=[1,16]\setminus\{4,8,12,16\}$.
We divide the proof into two cases.

Case 1. Suppose that each row of $A$ contains an element equivalent to $0$ modulo $3$.
Clearly we can assume without loss of generality that $3 \in supp(\overline{R}_1)$, $6 \in supp(\overline{R}_2)$, $9 \in supp(\overline{R}_3)$, $15 \in supp(\overline{R}_4)$.
It follows that
\begin{align*}
 supp(\overline{R}_1) & \in \{\{3,1,2\}, \{3,2,5\}, \{3,7,10\}, \{3,10,13\},
\{3,11,14\}\};\\
supp(\overline{R}_2) & \in \{\{6,1,5\},\{6,1,7\},\{6,5,11\},\{6,7,13\}\};\\
supp(\overline{R}_3) & \in \{\{9,2,7\},\{9,1,10\},\{9,2,11\},\{9,5,14\}\};\\
supp(\overline{R}_4) & \in \{\{15,1,14\},\{15,2,13\},\{15,5,10\}\}.
\end{align*}
A simple direct check shows us that these conditions are compatible with $supp(A)=[1,16]\setminus\{4,8,12,16\}$ only if:
\begin{equation}\label{righe}
\begin{array}{lcl}
supp(\overline{R}_1)=\{3,10,13\}; &\qquad & supp(\overline{R}_2)=\{6,5,11\}; \\[3pt]
supp(\overline{R}_3)=\{9,2,7\}; &  & supp(\overline{R}_4)=\{15,1,14\}.
\end{array}
\end{equation}

We can also assume, up to permutations of the columns and changing signs, that $\E(\overline{R}_4)=(15,-14,-1)$ and that the cell $(4,4)$ is empty.
Since $-14\in \overline{C}_2$, we have that $\E(\overline{C}_2)\in \{\{3,11,-14\},\{5, 9, -14\}\}$.
We consider these two cases separately:
$$\mathrm{(a)}\quad A=\begin{array}{|r|r|r|r|}
\hline  & 3 &   &   \;\;\;\;  \\
\hline   & 11 &  &    \\
\hline   &  &   &  \\
\hline  15 & -14 & -1 & \\
\hline
\end{array}; \qquad \mathrm{(b)}\quad
A=\begin{array}{|r|r|r|r|}
\hline  & &   &   \;\;\;\;  \\
\hline   & 5 &  &    \\
\hline   & 9 &  &  \\
\hline  15 & -14 & -1 & \\
\hline
\end{array}.$$
\begin{itemize}
\item[(a)] Since the cell $(4,4)$ is empty it follows that the cell $(1,4)$ is not empty and that
$A[1,4]\in\{10,-13\}$.
It is easy to see that if $A[1,4]=10$ then it is not possible to complete the column $\overline{C}_4$.
Hence $A[1,4]=-13$ which implies  $\E(\overline{C}_4)=\{-13,6,7\}$, but now there is no possible way to complete the row $\overline{R}_2$.
\item[(b)] Using \eqref{righe}, it is not hard to see that $\E(\overline{C}_1)=\{-13,-2,15\}$ and
$\E(\overline{C}_3)=\{3,-2,-1\}$, which clearly is a contradiction.
\end{itemize}

Case 2. Suppose now that there is a row $\overline{R}_i$ of $A$ such that each of its elements is equivalent to $\pm1$ modulo $3$.
Clearly, we can assume without loss of generality that $\overline{R}_i=\overline{R}_1$, hence
 $\{3,6,9,15\}\subset supp(\overline{R}_2)\cup supp(\overline{R}_3)\cup supp(\overline{R}_4)$.
 Because of the pigeonhole principle there exists a row $\overline{R}_j$, with $j\neq 1$, whose support contains at least two elements among $\{3,6,9,15\}$.
 We can assume that $j=2$ and that the filled positions of $\overline{R}_2$ are $(2,1),(2,2)$ and $(2,3)$.
 Since the sum of the elements of $\overline{R}_2$ is zero, we have that $|supp(\overline{R}_2)\cap \{3,6,9,15\}|=3$;
 let us denote by $x$ the element of $\{3,6,9,15\}$ that is not contained in $supp(\overline{R}_2)$.
 It follows that $x \in supp(\overline{C}_4)$, otherwise we would have a column with exactly two elements equivalent to $0$ modulo $3$,
  but this implies that the sum of this column is not zero. Therefore each column of $A$ contains an element equivalent to $0$ modulo $3$.
Now  reasoning as in the first case (on the columns instead of the rows) we obtain a contradiction.
\end{proof}

In this paper  we investigate the existence of an integer $\H_k(n;k)$.
Note that in this special case the necessary conditions given in Proposition \ref{prop:necc} can be written in a
simpler way. In fact, we are in case (1) with $t=k$ and hence we obtain the following result.

\begin{cor}\label{condnecc}
  If there exists an integer $\H_k(n;k)$, then necessarily
  one of the following holds:
  \begin{itemize}
    \item[(1)] $k$ is odd and $n\equiv 0,3\pmod 4$;
    \item[(2)] $k\equiv 2\pmod 4$ and $n$ is even;
    \item[(3)] $k\equiv 0\pmod 4$.
  \end{itemize}
\end{cor}

\section{Extension theorem} \label{sec:extension}

Firstly we introduce notations and definitions useful to present the main result
of this section which allows us to obtain an integer $\H_{k+h}(n;k+h)$ starting from an integer
$\H_k(n;k)$, for suitable even $h$. Above all, this result plays a crucial role in the paper.

\begin{defi}
A square p.f. array $A$ with entries in $\mathbb{Z}$ is said to be \emph{shiftable} if
 every row and every column contains an equal number of positive and negative entries.
\end{defi}
  Let $A$ be a shiftable array and $x$ a non-negative integer.
  Let $A\pm x$ be the array obtained adding $x$ to each positive entry of $A$ and $-x$    to each negative entry of $A$.

\begin{rem}\label{rem:shiftable}
  If $A$ is shiftable then the row and column sums of $A\pm x$ are exactly
   the row and column sums of $A$.
\end{rem}

If $A$ is an $n\times n$ p.f. array, for $i\in[1,n]$ we define the $i$-th diagonal
$$D_i=\{(i,1),(i+1,2),\ldots,(i-1,n)\}.$$
Here all the arithmetic on the row and the column indices is performed modulo $n$, where the set of reduced residues is $\{1,2,\ldots,n\}$.
We say that the diagonals $D_i,D_{i+1},\ldots, D_{i+k}$ are $k+1$ \emph{consecutive diagonals}.

\begin{defi}
  Let $k\geq 1$ be an integer. We will say that a square p.f. array $A$ of size $n\geq k$ is \emph{cyclically} $k$-\emph{diagonal}
  if the non empty cells of $A$ are exactly those of $k$ consecutive diagonals.
\end{defi}

\begin{defi}
We call \emph{cyclically} $(s,k)$-\emph{diagonal} every  p.f. array $S$ obtained as follows. Take a cyclically
$k$-diagonal $n\times n$ p.f.  array $A$ and
replace each cell of $A$ with an $s\times s$ array which is totally empty if the corresponding cell of $A$ is empty.
Denote by $B$ the $sn \times sn$ array so obtained and let $\overline{C}_1,\ldots,\overline{C}_{sn}$ be its columns.
Let $S$ be any array whose ordered columns
are $\overline{C}_i,\overline{C}_{i+1}, \ldots,\overline{C}_{sn}, \overline{C}_1,\ldots, \overline{C}_{i-1}$, with $i\in [1,sn]$.
\end{defi}

\begin{ex}\label{ex:165}
The following is a cyclically $(2,3)$-diagonal array of size $8$.
 \begin{center}
\begin{footnotesize}
$\begin{array}{|rr|rr|rr|rr|}
\hline 4 &  & 36  & -28 &  &  & -33 & 21   \\
   & 8 & -27 & 39 &  &  & 20 & -40\\
\hline  -22 & 13 & 3 &  & -35  & 41 &  &   \\
  12 &  -29 &  & 7 & 42 & -32 &  &  \\
\hline   &  & 26 & -37  & 1 &  & -14 & 24 \\
    &  & -38 & 19 &  & 5 & 25 & -11 \\
\hline   15 & -10 &  &  & 23 & -30  & 2 &   \\
  -9 & 18 &  &   & -31 & 16 &  & 6  \\
\hline
\end{array}$
\end{footnotesize}
\end{center}
\end{ex}

\begin{rem}\label{cyc}
Each cyclically $k$-diagonal p.f. array of even size $n$ with $k$ odd can be viewed
as a cyclically $\left(2,\frac{k+1}{2} \right)$-diagonal p.f. array.
\end{rem}

\begin{ex}\label{123}
The following p.f. array is a cyclically $3$-diagonal integer $\H_3(12;3)$ whose filled diagonals are $D_{12},D_1,D_2$. This array can be also viewed as a cyclically $(2,2)$-diagonal.
  \begin{center}
\begin{footnotesize}
$\begin{array}{|r|rr|rr|rr|rr|rr|r|}
\hline -5 & 18 &   &  &  &  &  & & & &  & -13 \\
  -32 & 1 & 31 &  &  &  &  &  & & & & \\
\hline   & -19 & 2 & 17 &   &  &  & & & & & \\
   &   & -33 & 3 & 30 &  &  &  & & & & \\
\hline   &  &  &   -20 & 4 & 16  & & & & & & \\
   &  &  &   &   -34 & 12 & 22 & & & & & \\
\hline    &  &  &  &  & -28 & 7 & 21 &  & & & \\
   &  &  &  &  &   & -29 & -6  & 35 &  & & \\
\hline   &  &  &   &  &  &  & -15 & -8 & 23 &  & \\
    &  &  &  &  &  &  &  & -27 & -9 & 36 & \\
\hline   &  &  &  &  &   &  &   & & -14 & -10 & 24 \\
  37 &  &  &  &  &   &  &   & &  & -26 & -11 \\
\hline
\end{array}$
\end{footnotesize}
\end{center}
\end{ex}

\begin{thm}\label{thm:Ext4}
Suppose there exists an integer $\H_k(n;k)$, say $A$, and an $n\times n$ shiftable p.f. array $B$ such that:
\begin{itemize}
\item[(1)] each row and each column of $B$ contains $h$ filled cells;
\item[(2)] $supp(B)=\left[1,\frac{h}{2}(2n+1)\right]\setminus T$ where:
$$T=\begin{cases} \left\{(2n+1),2(2n+1),\dots,\frac{h}{2}(2n+1)\right\} \mbox{ if }k \mbox{ is even;}\\\left\{(n+1),(n+1)+(2n+1),\dots,(n+1)+\frac{h-2}{2}(2n+1)\right\} \mbox{ if }k\mbox{ is odd;}\end{cases}$$
\item[(3)] the elements in every row and column of $B$ sum to $0$;
\item[(4)] $skel(A)\cap skel(B)=\emptyset$.
\end{itemize}
Then there exists an integer $\H_{k+h}(n;k+h)$.
\end{thm}

\begin{proof}
Note that since $B$ is shiftable then $h$ is even.
We divide the proof into two cases according to the parity of $k$.

Case 1: $k$ is even. Since $A$ is an integer $\H_k(n;k)$, we have that:
\[
supp(A)=\left[1,\frac{k}{2}(2n+1)\right]\setminus \left\{(2n+1),2(2n+1),\dots,\frac{k}{2}(2n+1)\right\}.
\]
Since $B$ is shiftable, by  Remark \ref{rem:shiftable} and by $(3)$, the rows and columns of
 $\overline{B}=B\pm\frac{k}{2}(2n+1)$ still sum to zero.
Moreover, because of hypothesis $(2)$, we also have that:
  \[
 supp(\overline{B})=\left[\frac{k}{2}(2n+1)+1,\frac{k+h}{2}(2n+1)\right]\setminus \left\{\frac{k+2}{2}(2n+1),\dots,\frac{k+h}{2}(2n+1)\right\}.
\]
It follows from hypotheses $(1)$ and $(4)$ that the union of $A$ and $\overline{B}$ is an integer $\H_{k+h}(n;k+h)$.

Case 2: $k$ is odd. We proceed in a similar way. Here we have that:
\[
supp(A)=\left[1,\frac{k(2n+1)-1}{2}\right]\setminus \left\{(2n+1),\dots,\frac{k-1}{2}(2n+1)\right\}.
\]
Since $B$ is shiftable, by  Remark \ref{rem:shiftable} and by $(3)$, the rows and columns of
 $\overline{B}=B\pm\frac{k(2n+1)-1}{2}$ still sum to zero.
Moreover, because of hypothesis $(2)$, we have that the support of $\overline{B}$ is:
\[
\left[\frac{k(2n+1)+1}{2},\frac{(k+h)(2n+1)-1}{2}\right]\setminus \left\{\frac{k+1}{2}(2n+1),\dots,\frac{k+h-1}{2}(2n+1)\right\}.
\]
It follows from hypotheses $(1)$ and $(4)$ that the union of $A$ and $\overline{B}$ is an integer $\H_{k+h}(n;k+h)$.
\end{proof}

Many of the constructions we will present  are based on filling in the cells of a set of diagonals.
In order to describe these constructions we use the same procedure introduced in \cite{DW}.
In an $n\times n$ array $A$ the procedure $diag(r,c,s,\Delta_1,\Delta_2,\ell)$ installs the entries
\[
A[r+i\Delta_1,c+i\Delta_1]=s+i\Delta_2\qquad \textrm{for}\ i\in[0,\ell-1].
\]
The parameters used in the $diag$ procedure have the following meaning:
\begin{itemize}
  \item $r$ denotes the starting row,
  \item $c$ denotes the starting column,
  \item $s$ denotes the entry $A[r,c]$,
  \item $\Delta_1$ denotes the increasing value of the row and column at each step,
  \item $\Delta_2$ denotes how much the entry is changed at each step,
  \item $\ell$ is the length of the chain.
\end{itemize}
We will write $[a,b]_{(W)}$ to mean $supp(W)=[a,b]$.

Here we provide some direct constructions of shiftable p.f. arrays that satisfy the hypotheses of Theorem \ref{thm:Ext4}.

\begin{prop}\label{prop:4}
There exists a shiftable, integer, cyclically $4$-diagonal $\H_4(n;4)$ for $n\geq4$.
\end{prop}

\begin{proof}
 We construct an $n\times n$ array $A$ using the following procedures labeled $\texttt{A}$ to $\texttt{D}$:
$$\begin{array}{lcl}
\texttt{A}:\;   diag(1,1,1,1,1,n);  & \hfill &
\texttt{B}:\; diag(1,2,-(n+1),1,-1,n);\\[3pt]
\texttt{C}:\;   diag(1,3,-(2n+4),1,-1,n-2); &  &
\texttt{D}:\;  diag(1,4,3n+4,1,1,n-2).
\end{array}$$
We also fill the following cells in an \textit{ad hoc} manner:
$$\begin{array}{lcl}
 A\left[n-1,1\right]=-(2n+2), & \quad & A\left[n-1,2\right]=3n+2, \\[3pt]
 A\left[n,2\right]=-(2n+3),  & &  A\left[n,3\right]=3n+3.
   \end{array}$$
We now prove that the array constructed above is an integer $\H_4(n;4)$.
To aid in the proof we give a schematic picture (see Figure \ref{fig1}) of where each of the diagonal procedures
fills cells.
We have placed an $\texttt{X}$ in the \textit{ad hoc} cells.
It is easy to see that $A$ is shiftable and cyclically $4$-diagonal. We now check that the elements in every row sum to $0$
(in $\Z$).

\begin{figure}
\begin{footnotesize}
 \begin{center}
\begin{tabular}{|c|c|c|c|c|c|c|}
  \hline
  \texttt{A} & \texttt{B} & \texttt{C} & \texttt{D} & & & \\ \hline
 & \texttt{A} & \texttt{B} & \texttt{C} & \texttt{D} & & \\ \hline
 & & \texttt{A} & \texttt{B} & \texttt{C} & \texttt{D} & \\ \hline
 & & & \texttt{A} & \texttt{B} & \texttt{C} & \texttt{D} \\ \hline
 \texttt{D} & & & & \texttt{A} & \texttt{B} & \texttt{C} \\ \hline
 \texttt{X} & \texttt{X} & & & & \texttt{A} & \texttt{B} \\ \hline
 \texttt{B} & \texttt{X} & \texttt{X} & & & & \texttt{A}\\ \hline
 \end{tabular}
\end{center}
\end{footnotesize}
\caption{Scheme of construction with $n=7$.}\label{fig1}
\end{figure}

\begin{description}
\item[Row $1$ to $n-2$] Notice that for any row  $r=1+i$, where $i\in\left[0,n-3\right]$, from the
\texttt{A},
\texttt{B},  \texttt{C} and \texttt{D} diagonal cells we get the following sum:
$$(1+i)-(n+1+i) -(2n+4+i) + (3n+4+i)=0.$$
\item[Row $n-1$] This row contains two \textit{ad hoc} values, the $(n-1)$-th element of the \texttt{A} diagonal and
the
$(n-1)$-th element of the \texttt{B} diagonal. The sum is $$-(2n+2) + (3n+2) + (n-1)  -(2n-1)=0.$$
 \item[Row $n$] This row contains two \textit{ad hoc} values, the last  of the \texttt{A} diagonal and the last of the
\texttt{B} diagonal. The sum is $$-2n - (2n+3) + (3n+3) +n=0.$$
   \end{description}
So we have shown that all row sums are zero. Next we check that the columns all add to zero.
\begin{description}
\item[Column $1$] There is an \textit{ad hoc} value plus the first of the \texttt{A} diagonal as well as the last
elements  of the \texttt{D} and \texttt{B} diagonals. The sum is $$1+(4n+1)-(2n+2)-2n=0.$$
\item[Column $2$] There are two \textit{ad hoc} values plus the first of the \texttt{B} diagonal as well as the second
 of the \texttt{A} diagonal. The sum is $$-(n+1)+2+(3n+2)-(2n+3)=0.$$
\item[Column $3$] There is an  \textit{ad hoc} value plus the first of the \texttt{C} diagonal,  the second  of the
\texttt{B} diagonal and the third  of the \texttt{A} diagonal. The sum is $$-(2n+4)-(n+2)+3+(3n+3)=0.$$
\item[Column $4$ to $n$]  For every column $c$, write $c=4+i$, where $i\in\left[0,n-4\right]$. From the \texttt{D},
\texttt{C},  \texttt{B} and \texttt{A} diagonal cells we get the following sum:
$$(3n+4+i)-(2n+5+i)-(n+3+i)+(4+i)=0.$$
\end{description}
So we have shown that each column sums to $0$. Now we consider the support of $A$:
$$\begin{array}{rcl}
supp(A) &= & [1,n]_{(\texttt{A})}\cup [n+1,2n]_{(\texttt{B})}\cup \{ 2n+2, 2n+3\} \cup
[2n+4,3n+1]_{(\texttt{C})}\cup \\[3pt]
&& \{3n+2, 3n+3\}\cup [3n+4,4n+1]_{(\texttt{D})}\\[3pt]
& = & [1,4n+1]\setminus \{2n+1\}.
\end{array}$$
Thus, $A$ is a shiftable, integer, cyclically $4$-diagonal $\H_4(n;4)$ for $n\geq 4$.
\end{proof}

\begin{ex}
Following the proof of Proposition \ref{prop:4} we obtain the integer $\H_4(7;4)$ below.
  \begin{center}
\begin{footnotesize}
$\begin{array}{|r|r|r|r|r|r|r|}
\hline 1 & -8 & -18  & 25 &  & &     \\
\hline   & 2 & -9 & -19 & 26 & &   \\
\hline   &  & 3 & -10 & -20  & 27 &  \\
\hline   &   &  & 4 & -11 & -21 & 28 \\
\hline  29 &  &  &   &   5 & -12 & -22\\
\hline  -16  & 23 &  &  &  & 6 & -13\\
\hline  -14  & -17 & 24 &  &  &  & 7\\
\hline
\end{array}$
\end{footnotesize}
\end{center}
\end{ex}

\begin{prop}\label{prop:4B}
For every $n\geq 4$, there exists an $n\times n$ shiftable, cyclically $4$-diagonal, p.f. array $B$ such that:
\begin{itemize}
\item[(1)] $supp(B)=[1,4n+2]\setminus \{n+1,3n+2\}$;
\item[(2)] the elements in every row and column of $B$ sum to $0$.
\end{itemize}
\end{prop}

\begin{proof}
 We construct an $n\times n$ array $B$ using the following procedures labeled $\texttt{A}$ to $\texttt{D}$:
$$\begin{array}{lcl}
\texttt{A}:\; diag(1,1,1,1,1,n); & \hfill &
\texttt{B}:\; diag(1,2,-(n+2),1,-1,n);\\[3pt]
\texttt{C}:\; diag(1,3,-(2n+4),1,-1,n-2);&&
\texttt{D}:\; diag(1,4,3n+5,1,1,n-2).
  \end{array}$$
We also fill the following cells in an \textit{ad hoc} manner:
$$\begin{array}{lcl}
 B\left[n-1,1\right]=-(2n+2) ,& \quad & B\left[n-1,2\right]=3n+3,\\[3pt]
 B\left[n,2\right]=-(2n+3), &  & B\left[n,3\right]=3n+4.
   \end{array}$$
To aid in the proof we give a schematic picture of where each of the diagonal procedures fills cells (see Figure \ref{fig1}).
We have placed an $\texttt{X}$ in the \textit{ad hoc} cells.
Note that $B$ is shiftable and cyclically $4$-diagonal, so we only need to prove that the array constructed above
satisfies the properties $(1)$ and $(2)$ of the statement. We now check that the elements in every row sum to $0$  (in $\Z$).

\begin{description}
\item[Row $1$ to $n-2$] Notice that for any row $r=1+i$, where $i\in\left[0,n-3\right]$, from the
\texttt{A},
\texttt{B},  \texttt{C} and \texttt{D} diagonal cells we get the following sum:
$$(1+i)-(n+2+i) -(2n+4+i) + (3n+5+i)=0.$$
\item[Row $n-1$] This row contains two \textit{ad hoc} values, the $(n-1)$-th element of the \texttt{A} diagonal and
the
$(n-1)$-th element of the \texttt{B} diagonal. The sum is $$-(2n+2) + (3n+3) + (n-1)  -2n=0.$$
 \item[Row $n$] This row contains two \textit{ad hoc} values and the last elements  of the \texttt{A} and
\texttt{B} diagonals. The sum is $$-(2n+1) - (2n+3) + (3n+4) +n=0.$$
   \end{description}
So we have shown that all row sums are zero. Next we check that the columns all add to zero.
\begin{description}
\item[Column $1$] There is an \textit{ad hoc} value plus the first of the \texttt{A} diagonal as well as the last
elements  of the \texttt{D} and \texttt{B} diagonals. The sum is $$1+(4n+2)-(2n+2)-(2n+1)=0.$$
\item[Column $2$] There are two \textit{ad hoc} values plus the first of the \texttt{B} diagonal as well as the second
 of the \texttt{A} diagonal. The sum is $$-(n+2)+2+(3n+3)-(2n+3)=0.$$
\item[Column $3$] There is an  \textit{ad hoc} value plus the first of the \texttt{C} diagonal,  the second  of the
\texttt{B} diagonal and the third  of the \texttt{A} diagonal. The sum is $$-(2n+4)-(n+3)+3+(3n+4)=0.$$
\item[Column $4$ to $n$]  Notice that for every column  $c=4+i$, where $i\in\left[0,n-4\right]$, from the \texttt{D},
\texttt{C},  \texttt{B} and \texttt{A} diagonal cells we get the following sum:
$$(3n+5+i)-(2n+5+i)-(n+4+i)+(4+i)=0.$$
\end{description}
So we have shown that each column sums to $0$. Now we consider the support of $B$:
$$\begin{array}{rcl}
supp(B) & =& [1,n]_{(\texttt{A})}\cup [n+2,2n+1]_{(\texttt{B})}\cup \{2n+2, 2n+3\}\cup
\\[3pt]
&& [2n+4,3n+1]_{(\texttt{C})}\cup \{3n+3, 3n+4\}
\cup[3n+5,4n+2]_{(\texttt{D})}\\[3pt]
& =&[1,4n+2]\setminus \{n+1,3n+2\}.
\end{array}$$
Hence we obtain the result.
\end{proof}

\begin{ex}
Here we have the $7\times 7$ array obtained following the proof of Proposition \ref{prop:4B}.
  \begin{center}
\begin{footnotesize}
$\begin{array}{|r|r|r|r|r|r|r|}
\hline 1 & -9& -18  & 26 &  & &     \\
\hline   & 2 & -10 & -19 & 27 & &   \\
\hline   &  & 3 & -11 & -20  & 28 &  \\
\hline   &   &  & 4 & -12 & -21 & 29 \\
\hline  30 &  &  &   &   5 & -13 & -22\\
\hline  -16  & 24 &  &  &  & 6 & -14\\
\hline  -15  & -17 & 25 &  &  &  & 7\\
\hline
\end{array}$
\end{footnotesize}
\end{center}
\end{ex}

\begin{prop}\label{prop:4B2}
There exists a shiftable, integer, cyclically $(2,2)$-diagonal $\H_4(n;4)$ for any even $n\geq 4$.
\end{prop}

\begin{proof}
We set $n=2m$. Let us consider the arrays
$E_i=\begin{array}{|c|c|}\hline
 1+4i & -2-4i \\\hline
 -3-4i & 4+4i \\\hline
 \end{array}\;$
and
$F_i=\begin{array}{|c|c|}\hline
 -2-4i & 3+4i \\\hline
 4+4i & -5-4i \\\hline
 \end{array}\;$.
Now, let $B$ be the $2m\times 2m$ array so defined:
\[
\begin{array}{|c|c|c|c|c|c|c|}
  \hline
 E_0 & \;\;\;F_m\;\;\; &  & & & & \\ \hline
 & E_1 & F_{m+1} &   & & &   \\ \hline
 & & E_2 & F_{m+2} &    & &   \\ \hline
 & & & \ddots & \;\;\ddots\;\; & &      \\ \hline
& & & &\ddots & \;\;\ddots\;\; &    \\ \hline
 & & & & &  E_{m-2} & F_{2m-2}     \\ \hline
F_{2m-1} & & & & &  & E_{m-1}     \\ \hline
\end{array}
\]
Clearly $B$ is shiftable and cyclically $(2,2)$-diagonal; its support is given by:
$$\begin{array}{rcl}
supp(B) & = & \bigcup_{i=0}^{m-1} supp(E_i)\cup\bigcup_{i=m}^{2m-1} supp(F_i)\\
& = & [1,4m]\cup[4m+2,8m+1]=[1,4n+1]\setminus \{2n+1\}.
  \end{array}$$
It is also easy to check that each row and each column of $B$ sums to zero and thus $B$ is a $\H_4(n;4)$ that satisfies the required properties.
\end{proof}

\begin{ex}
Following the proof of Proposition \ref{prop:4B2} we obtain the integer $\H_4(8;4)$ below:
  \begin{center}
\begin{footnotesize}
$\begin{array}{|rr|rr|rr|rr|}
\hline 1 & -2& -18  &19 &  && &     \\
  -3 & 4 & 20 & -21& && &   \\
\hline   &  & 5 & -6 & -22  & 23 & & \\
   &   & -7 & 8 & 24& -25 & &\\
\hline   &  &  &   & 9 & -10& -26& 27\\
   &  &  &   & -11 & 12& 28& -29\\
\hline  -30 &31& &  &  &  & 13& -14\\
  32 &-33& &  &  &  & -15& 16\\
\hline
\end{array}$
\end{footnotesize}
\end{center}
\end{ex}

Given a cyclically $(2,3)$-diagonal p.f. array, we call strip $S_i$ the union of two consecutive rows
$\overline{R}_{2i+1}$ and $\overline{R}_{2i+2}$.

\begin{prop}\label{prop:6Finta}
Let $n\equiv 0 \pmod{4}$ with $n\geq8$. Then there exists an $n\times n$ shiftable, cyclically $(2,3)$-diagonal,
p.f. array $B$ such that:
\begin{itemize}
\item[(1)] $supp(B)=[1,6n+3]\setminus \left\{n+1,3n+2,5n+3\right\};$
\item[(2)] the elements in every row and column of $B$ sum to $0$.
\end{itemize}
\end{prop}

\begin{proof}
Consider the following three $2\times 6$ arrays:
$$
U=\begin{array}{|c|c|c|c|c|c|c|}\hline
-1 & 5 & 2 & -7  & -9 & 10\\\hline
3 & -4 & -6 & 8 & 11 & -12\\\hline
\end{array}\;,
$$
$$V_5=\begin{array}{|c|c|c|c|c|c|}\hline
-1 & 10 & 7 & -12 & 4 & -8\\\hline
3 & -9 & -11 & 13 & -2 & 6\\\hline
    \end{array}\;,$$
$$V_9=\begin{array}{|c|c|c|c|c|c|}\hline
-1 & 5 & 2 & -7 & 13 & -12\\\hline
3 & -4 & -6 & 8 & -11 & 10\\\hline
    \end{array}\;.$$
Note that $supp(U)=[1,12]$, $supp(V_5)=[1,4]\cup [6,13]$ and
$supp(V_9)=[1,8]\cup [10,13]$. Also,  $U$, $V_5$ and $V_9$ are shiftable matrices whose rows sum to $0$ and
whose columns have the following sums: $(2,1,-4,1,2,-2)$.
As consequence, every cyclically $(2,3)$-diagonal p.f. array constructed strip by strip using $2\times 6$ arrays of the form $W\pm x$, where $W\in \{U,V_5,V_9\}$
and $x$ is a non-negative integer, has rows and columns that sum to $0$. We have to distinguish three
cases.

If $n=12m$, let $B$ be a cyclically $(2,3)$-diagonal $n\times n$ p.f. array whose strips $S_i$ are as follows:
$$S_i =\left\{\begin{array}{ll}
U\pm 12i & \textrm{ if }  i \in [0,m-1],\\
U\pm (1+12i) & \textrm{ if }  i \in [m,3m-1],\\
U\pm (2+12i) & \textrm{ if }  i \in [3m,5m-1],\\
U\pm (3+12i) & \textrm{ if }  i \in [5m,6m-1].\\
 \end{array}\right.$$
We obtain
$$\begin{array}{rcl}
supp(B) & =& \bigcup_{i=0}^{m-1}[1+12i,12+12i] \cup \bigcup_{i=m}^{3m-1}[2+12i,13+12i]\cup \\[3pt]
&& \bigcup_{i=3m}^{5m-1}[3+12i, 14+12i] \cup \bigcup_{i=5m}^{6m-1} [4+12i,15+12i]\\[3pt]
& =& [1,12m]\cup [12m+2, 36m+1]\cup [36m+3, 60m+2]\cup\\[3pt]
&& [60m+4, 72m+3].
\end{array}$$

If $n=12m+4$, let $B$ be a cyclically $(2,3)$-diagonal  $n\times n$ p.f. array whose strips $S_i$ are:
$$S_i =\left\{\begin{array}{ll}
U\pm 12i & \textrm{ if }  i \in [0,m-1],\\
V_5\pm 12i & \textrm{ if }  i =m,\\
U\pm (1+12i) & \textrm{ if }  i \in [m+1,3m],\\
U\pm (2+12i) & \textrm{ if }  i \in [3m+1,5m],\\
V_9\pm (2+12i) & \textrm{ if }  i =5m+1,\\
U\pm (3+12i) & \textrm{ if }  i \in [5m+2,6m+1].\\
 \end{array}\right.$$
It follows that
$$\begin{array}{rcl}
supp(B) & =& \bigcup_{i=0}^{m-1}[1+12i,12+12i]\cup [12m+1,12m+4]\cup \\[3pt]
&& [12m+6,12m+13]\cup \bigcup_{i=m+1}^{3m}[2+12i, 13+12i] \cup\\[3pt]
&& \bigcup_{i=3m+1}^{5m}[3+12i, 14+12i]\cup [60m+15,60m+22]\cup\\[3pt]
&& [60m+24,60m+27]\cup \bigcup_{i=5m+2}^{6m+1} [4+12i,15+12i]\\[3pt]
& =& [1,12m+4]\cup [12m+6,36m+13]\cup [36m+15, 60m+22]\cup\\
&& [60m+24, 72m+27].
\end{array}$$

If $n=12m+8$ with $m\geq 0$, let $B$ be a cyclically $(2,3)$-diagonal $n\times n$ p.f. array whose strips $S_i$ are:
$$S_i =\left\{\begin{array}{ll}
U\pm 12i & \textrm{ if }  i \in [0,m-1],\\
V_9\pm 12i & \textrm{ if }  i =m,\\
U\pm (1+12i) & \textrm{ if }  i \in [m+1,3m+1],\\
U\pm (2+12i) & \textrm{ if }  i \in [3m+2,5m+2],\\
V_5\pm (2+12i) & \textrm{ if }  i =5m+3,\\
U\pm (3+12i) & \textrm{ if }  i \in [5m+4,6m+3].\\
 \end{array}\right.$$
It follows that
$$\begin{array}{rcl}
supp(B) & =& \bigcup_{i=0}^{m-1}[1+12i,12+12i]\cup [12m+1,12m+8]\cup \\[3pt]
&& [12m+10,12m+13] \cup \bigcup_{i=m+1}^{3m+1}[2+12i,13+12i] \cup\\[3pt]
&& \bigcup_{i=3m+2}^{5m+2}[3+12i,14+12i]\cup [60m+39, 60m+42]\cup \\[3pt]
&& [60m+44, 60m+51 ]\cup \bigcup_{i=5m+4}^{6m+3}[4+12i, 15+12i]\\[3pt]
& = & [1,12m+8]\cup [12m+10, 36m+25] \cup [36m+27, 60m+42] \cup\\[3pt]
&& [60m+44, 72m+51  ].
\end{array}$$

In all three cases, we have $supp(B)=[1,6n+3]\setminus \left\{n+1, 3n+2 ,5n+3\right\}$ as required.
\end{proof}

\begin{ex}\label{414}
Following the proof of Proposition \ref{prop:6Finta} we obtain the following
$12\times 12$ p.f. array $B$:
  \begin{center}
\begin{footnotesize}
$\begin{array}{|rr|rr|rr|rr|rr|rr|}\hline
-1 & 5 & 2 &  -7 &  -9 & 10 & & && &&\\
3 & -4 & -6 & 8 & 11 & -12 & & && & &\\\hline
& & -14 & 18 &  15   & -20 & -22 & 23 &  &  &  & \\
 &  & 16 & -17 & -19 & 21 & 24 & -25 &  &  &  & \\\hline
 &  &  &  & -26 & 30 & 27 & -32 & -34 & 35 &  & \\
 &  &  &  & 28 & -29 & -31 & 33 & 36 & -37 &  & \\\hline
 &  &  &  &  &  & -39 & 43 & 40 & -45 & -47 & 48\\
 &  &  &  &  &  & 41 & -42 & -44 & 46 & 49 & -50\\\hline
-59 & 60 &  &  &  &  &  &  & -51 & 55 & 52 & -57\\
61 & -62 &  &  &  &  &  &  & 53 & -54 & -56 & 58\\\hline
65 & -70 & -72 & 73 &  &  &  &  &  &  & -64 & 68\\
-69 & 71 & 74 & -75 &  &  &  &  &  &  & 66 & -67\\ \hline
\end{array}$
\end{footnotesize}
\end{center}
\end{ex}

\begin{thm}\label{prop:ExtDiag}
If there exists:
\begin{itemize}
  \item[(1)] an integer cyclically $k$-diagonal $\H_k(n;k)$ with $n\geq k+4$,
  then there exists an integer cyclically $(k+4)$-diagonal $\H_{k+4}(n;k+4)$;
  \item[(2)] an integer cyclically $(2,d)$-diagonal $\H_{2d}(n;2d)$ with $n\geq 2d+4$, then
  there exists an integer cyclically $(2,d+2)$-diagonal
	$\H_{2d+4}(n;2d+4)$;
  \item[(3)] an integer cyclically $(2,d)$-diagonal $\H_{2d-1}(n;2d-1)$ with $n\equiv 0 \pmod {4}$, $n\geq 2d+6$,
  then there exists an integer cyclically $(2,d+3)$-diagonal $\H_{2d+5}(n;2d+5)$.
\end{itemize}
\end{thm}

\begin{proof}
(1) Let $A$ be an integer cyclically $k$-diagonal $\H_k(n;k)$ with $n\geq k+4$.
Let $B$ be the cyclically $4$-diagonal array of size $n$ constructed in Proposition \ref{prop:4}
if $k$ is even and in Proposition \ref{prop:4B} if $k$ is odd.
Since  $n\geq k+4$, starting from $B$ it is possible to construct a cyclically $4$-diagonal array $\tilde B$ such that $skel(A) \cap skel(\tilde B)=\emptyset$ and $A\cup \tilde B$ is cyclically $(k+4)$-diagonal.
Hence the result follows by Theorem \ref{thm:Ext4}.

(2) Let $A$ be an integer cyclically $(2,d)$-diagonal $\H_{2d}(n;2d)$ with $n\geq 2d+4$.
Let $B$ be the cyclically $(2,2)$-diagonal array constructed in Proposition \ref{prop:4B2}.
Reasoning as in case (1), since $n\geq 2d+4$ we can take $\tilde B$ such that $skel(A) \cap skel(\tilde B)=\emptyset$ and $A \cup \tilde B$ is  cyclically $(2,d+2)$-diagonal. Hence the result follows by Theorem \ref{thm:Ext4}.

  (3) Let $A$ be an integer cyclically $(2,d)$-diagonal $\H_{2d-1}(n;2d-1)$ with $n\equiv 0 \pmod 4$ and $n\geq 2d+6$.
  Let $B$ be  the cyclically $(2,3)$-diagonal array constructed in Proposition \ref{prop:6Finta}.
  Reasoning as in case (1), since $n\geq 2d+6$ we can take $\tilde B$ such that $skel(A) \cap skel(\tilde B)=\emptyset$
	and $A\cup \tilde B$ is cyclically 	$(2,d+3)$-diagonal.
  Hence the result follows by Theorem \ref{thm:Ext4}.
\end{proof}

\begin{ex}
We consider case (3) of previous theorem for $n=12$ and $d=2$.
Taking the integer cyclically $(2,2)$-diagonal $\H_{3}(12;3)$ of Example \ref{123}
and the cyclically $(2,3)$-diagonal array $B$ of Example \ref{414}, we obtain an integer cyclically $(2,5)$-diagonal $\H_{9}(12;9)$.
The elements in bold are those of $\tilde B\pm 37$.
  \begin{center}
\begin{footnotesize}
$\begin{array}{|r|r|r|r|r|r|r|r|r|r|r|r|}\hline
 -5 &  18 &   &  \mathbf{-38} &  \mathbf{42} &  \mathbf{39} &  \mathbf{-44} &  \mathbf{-46} &  \mathbf{47} &   &   &  -13 \\\hline
 -32 &  1 &  31 &  \mathbf{40} &  \mathbf{-41} &  \mathbf{-43} &  \mathbf{45} &  \mathbf{48} &  \mathbf{-49} &   &   &  \\ \hline
  &  -19 &  2 &  17 &   & \mathbf{ -51} &  \mathbf{55} &  \mathbf{52} & \mathbf{ -57} &  \mathbf{-59} &  \mathbf{60} &   \\ \hline
  &   &  -33 &  3 &  30 &  \mathbf{53} &  \mathbf{-54} &  \mathbf{-56} &  \mathbf{58} &  \mathbf{61} &  \mathbf{-62} &   \\ \hline
 \mathbf{72} &   &   &  -20 &  4 &  16 &   &  \mathbf{-63} & \mathbf{ 67} &  \mathbf{64} &  \mathbf{-69} &  \mathbf{-71} \\ \hline
 \mathbf{-74} &   &   &   &  -34 &  12 &  22 &  \mathbf{65} &  \mathbf{-66} &  \mathbf{-68} &  \mathbf{70} &  \mathbf{73} \\ \hline
 \mathbf{-82} &  \mathbf{-84} &  \mathbf{85} &   &   &  -28 &  7 &  21 &   &  \mathbf{-76} &  \mathbf{80} &  \mathbf{77} \\ \hline
 \mathbf{83 }&  \mathbf{86} &  \mathbf{-87} &   &   &   &  -29 &  -6 &  35 &  \mathbf{78} &  \mathbf{-79} &  \mathbf{-81} \\ \hline
 \mathbf{92} &  \mathbf{89 }&  \mathbf{-94} &  \mathbf{-96} &  \mathbf{97} &   &   &  -15 &  -8 &  23 &   &  \mathbf{-88} \\ \hline
 \mathbf{-91} &  \mathbf{-93} &  \mathbf{95} &  \mathbf{98} &  \mathbf{-99} &   &   &   &  -27 &  -9 &  36 &  \mathbf{90} \\ \hline
  &  \mathbf{-101} &  \mathbf{105} &  \mathbf{102} & \mathbf{ -107} &  \mathbf{-109} &  \mathbf{110} &   &   &  -14 &  -10 &  24 \\\hline
 37 &  \mathbf{103} &  \mathbf{-104} &  \mathbf{-106} &  \mathbf{108} &  \mathbf{111} &  \mathbf{-112} &   &   &   &  -26 &  -11 \\
\hline
\end{array}$
\end{footnotesize}
\end{center}
\end{ex}

\section{Direct constructions of $\H_k(n;k)$}\label{sec:constructions}

In this section we give direct constructions of integer  $\H_k(n;k)$ with
$k=3,5,6$, since the case $k=4$ has been already considered in Proposition \ref{prop:4}.

\begin{prop}\label{prop:3odd}
There exists an integer cyclically $3$-diagonal $\H_3(n;3)$ for $n\equiv 3 \pmod 4$.
\end{prop}

\begin{proof}
We construct an $n\times n$ array $A$ using the following procedures labeled $\texttt{A}$ to $\texttt{J}$:
$$\begin{array}{ll}
\texttt{A}:\;   diag\left(2,2,1,1,1,\frac{n-3}{2}\right); &
\texttt{B}:\;   diag\left(\frac{n+3}{2},\frac{n+3}{2},-\frac{n+1}{2},1, -1,\frac{n-1}{2}\right);\\[3pt]
\texttt{C}:\;   diag\left(2,1,-\frac{5n+3}{2},2,-1,\frac{n+1}{4}\right); &
\texttt{D}:\;   diag\left(3,2,-\frac{3n+3}{2},2,-1,\frac{n-3}{4}\right);\\[3pt]
\texttt{E}:\;   diag\left(1,2,\frac{3n+1}{2},2,-1,\frac{n+1}{4}\right); &
\texttt{F}:\;   diag\left(2,3,\frac{5n+1}{2},2,-1,\frac{n-3}{4}\right);\\[3pt]
\texttt{G}:\;   diag\left(\frac{n+1}{2},\frac{n+3}{2},\frac{7n+3}{4},2,1,\frac{n+1}{4}\right); &
\texttt{H}:\;   diag\left(\frac{n+3}{2},\frac{n+5}{2},\frac{11n+7}{4},2,1,\frac{n+1}{4}\right);\\[3pt]
\texttt{I}:\;   diag\left(\frac{n+3}{2},\frac{n+1}{2},-\frac{9n+5}{4},2,1,\frac{n+1}{4}\right); &
\texttt{J}:\;   diag\left(\frac{n+5}{2},\frac{n+3}{2},-\frac{5n+1}{4},2,1,\frac{n+1}{4}\right).
\end{array}$$
We also fill the following cells in an \textit{ad hoc} manner:
$$\begin{array}{lcl}
 A\left[1,1\right]=-\frac{n-1}{2}, & \quad  & A\left[\frac{n+1}{2},\frac{n+1}{2}\right]=n.
  \end{array}$$

We now prove that the array constructed above is an integer  cyclically $3$-diagonal $\H_3(n;3)$.
To aid in the proof we give a schematic picture of where each of the diagonal procedures fills cells (see Figure \ref{fig2}).
We have placed an $\texttt{X}$ in the \textit{ad hoc} cells.
Note that each row and each column contains exactly $3$ elements. We now check that the elements in every row sum to $0$
(in $\Z$).

\begin{figure}[!ht]
\begin{footnotesize}
 \begin{center}
\begin{tabular}{|c|c|c|c|c|c|c|c|c|c|c|}
  \hline
 \texttt{X} & \texttt{E} &  & & & & & & &  & \texttt{J}\\ \hline
 \texttt{C} & \texttt{A} & \texttt{F} & & & & & & & &  \\ \hline
 &\texttt{D} & \texttt{A} & \texttt{E} & & & & & & &    \\ \hline
 & &\texttt{C} & \texttt{A} & \texttt{F}  & & & & & &   \\ \hline
 & & &\texttt{D} & \texttt{A} & \texttt{E} & & & & &    \\ \hline
 & & & &\texttt{C} & \texttt{X} & \texttt{G} & & & &    \\ \hline
 & & & & &\texttt{I} & \texttt{B} & \texttt{H} & & &     \\ \hline
 & & & & & &\texttt{J} & \texttt{B} & \texttt{G} & &    \\ \hline
 & & & & & & &\texttt{I} & \texttt{B} & \texttt{H} &    \\ \hline
 & & & & & & & &\texttt{J} & \texttt{B} & \texttt{G}    \\ \hline
\texttt{H} & & & & & & &  & & \texttt{I} & \texttt{B}     \\ \hline
\end{tabular}
\end{center}
 \end{footnotesize}
\caption{Scheme of construction with $n=11$.}\label{fig2}
\end{figure}

\begin{description}
\item[Row $1$] There is an \textit{ad hoc} value plus the first of the \texttt{E} diagonal as well as the last of the
\texttt{J} diagonal.
The sum is $$-\frac{n-1}{2}+\frac{3n+1}{2}-(n+1)=0.$$
\item[Row $2$  to $\frac{n-1}{2}$] There are two cases depending on whether the row $r$ is even or odd.
If $r$ is even, then write $r=2i+2$ where $i\in\left[0,\frac{n-7}{4}\right]$. Notice that from the \texttt{C},
\texttt{A} and \texttt{F} diagonal cells we get the following sum:
$$\left(-\frac{5n+3}{2}-i\right)+\left(1+2i\right)+\left(\frac{5n+1}{2}-i\right)=0.$$
If $r$ is odd, then write $r=2i+3$ where $i\in\left[0,\frac{n-7}{4}\right]$.
From the \texttt{D}, \texttt{A} and \texttt{E} diagonal cells we get the following sum:
$$\left(-\frac{3n+3}{2}-i\right)+\left(2+2i\right)+\left(\frac{3n-1}{2}-i\right)=0.$$
\item[Row $\frac{n+1}{2}$] We add the last value of the \texttt{C} diagonal, an  \textit{ad hoc} value and the first of
the \texttt{G} diagonal:
$$-\frac{11n+3}{4}+n+\frac{7n+3}{4}=0.$$
\item[Row $\frac{n+3}{2}$ to $n$] Note that $\frac{n+3}{2}$ is odd. There are two cases depending on whether the row $r$
is odd or even.
If $r$ is odd, then write $r=\frac{n+3}{2}+2i$ where $i\in\left[0,\frac{n-3}{4}\right]$.
Notice that from the \texttt{I}, \texttt{B} and \texttt{H} diagonal cells we get the following sum:
$$\left(-\frac{9n+5}{4}+i\right)+\left(-\frac{n+1}{2}-2i\right)+\left(\frac{11n+7}{4}+i\right)=0.$$
If $r$ is even, then write $r=\frac{n+5}{2}+2i$ where $i\in\left[0,\frac{n-7}{4}\right]$.
Notice that from the \texttt{J}, \texttt{B} and \texttt{G} diagonal cells we get the following sum:
$$\left(-\frac{5n+1}{4}+i\right)+\left(-\frac{n+1}{2}-2i\right)+\left(\frac{7n+3}{4}+i\right)=0.$$
\end{description}
So we have shown that all row sums are zero. Next we check that the columns all add to zero.
\begin{description}
\item[Column $1$] There is an \textit{ad hoc} value plus the first of the \texttt{C} diagonal as well as the last of the
\texttt{H} diagonal.
The sum is $$-\frac{n-1}{2}-\frac{5n+3}{2}+(3n+1)=0.$$
\item[Column $2$  to $\frac{n-1}{2}$] There are two cases depending on whether the column $c$ is even or odd.
If $c$ is even, then write $c=2i+2$ where $i\in\left[0,\frac{n-7}{4}\right]$.
Notice that from the \texttt{E}, \texttt{A} and \texttt{D} diagonal cells we get the following sum:
$$\left(\frac{3n+1}{2}-i\right)+(1+2i)+\left(-\frac{3n+3}{2}-i\right)=0.$$
If $c$ is odd, then write $c=2i+3$ where $i\in\left[0,\frac{n-7}{4}\right]$.
From the \texttt{F}, \texttt{A} and \texttt{C} diagonal cells we get the following sum:
$$\left(\frac{5n+1}{2}-i\right)+(2+2i)+\left(-\frac{5n+5}{2}-i\right)=0.$$
\item[Column $\frac{n+1}{2}$] We add the last value of the \texttt{E} diagonal, an  \textit{ad hoc} value and the first of
the \texttt{I} diagonal:
$$\frac{5n+5}{4}+n-\frac{9n+5}{4}=0.$$
\item[Column $\frac{n+3}{2}$ to $n$] Note that $\frac{n+3}{2}$ is odd.
There are two cases depending on whether the column $c$ is odd or even.
If $c$ is odd, then write $c=\frac{n+3}{2}+2i$ where $i\in\left[0,\frac{n-3}{4}\right]$.
Notice that from the \texttt{G}, \texttt{B} and \texttt{J} diagonal cells we get the following sum:
$$\left(\frac{7n+3}{4}+i\right)+\left(-\frac{n+1}{2}-2i\right)+\left(-\frac{5n+1}{4}+i\right)=0.$$
If $c$ is even, then write $c=\frac{n+5}{2}+2i$ where $i\in\left[0,\frac{n-7}{4}\right]$.
    Notice that from the \texttt{H}, \texttt{B} and \texttt{I} diagonal cells we get the following sum:
    $$\left(\frac{11n+7}{4}+i\right)+\left(-\frac{n+3}{2}-2i\right)+\left(-\frac{9n+1}{4}+i\right)=0.$$
\end{description}
So we have shown that each column sums to $0$. Now we consider the support of $A$:
$$\begin{array}{rcl}
supp(A)& = & \left[1,\frac{n-3}{2}\right]_{(\texttt{A})}\cup \{ \frac{n-1}{2}\} \cup
\left[\frac{n+1}{2},n-1\right]_{(\texttt{B})} \cup\{ n\} \cup \left[n+1,\frac{5n+1}{4}\right]_{(\texttt{J})}\cup \\[3pt]
&&\left[\frac{5n+5}{4},\frac{3n+1}{2}\right]_{(\texttt{E})}
 \cup \left[\frac{3n+3}{2},\frac{7n-1}{4}\right]_{(\texttt{D})} \cup \left[\frac{7n+3}{4},2n\right]_{(\texttt{G})}\cup\\[3pt]
 && \left[2n+2,\frac{9n+5}{4}\right]_{(\texttt{I})}\cup
 \left[\frac{9n+9}{4},\frac{5n+1}{2}\right]_{(\texttt{F})}\cup \left[\frac{5n+3}{2},\frac{11n+3}{4}\right]_{(\texttt{C})}
\cup \\[3pt]
&&\left[\frac{11n+7}{4},3n+1\right]_{(\texttt{H})}\\[3pt]
&= & [1,3n+1]\setminus \{2n+1\}.
\end{array}$$
\noindent Thus, $A$ is an integer cyclically 3-diagonal $\H_3(n;3)$ for $n\equiv 3\pmod 4$.
\end{proof}

\begin{ex}
Following the proof of Proposition \ref{prop:3odd} we obtain the
integer $\H_3(11;$ $3)$ below.
 \begin{center}
\begin{footnotesize}
$\begin{array}{|r|r|r|r|r|r|r|r|r|r|r|}
\hline -5 & 17 &   &  &  &  &  & & & & -12  \\
\hline  -29 & 1 & 28 &  &  &  &  &  & & & \\
\hline   & -18 & 2 & 16 &   &  &  & & & & \\
\hline   &   & -30 & 3 & 27 &  &  &  & & &\\
\hline   &  &  &   -19 & 4 & 15  & & & & &\\
\hline   &  &  &   &   -31 & 11 & 20 & & & & \\
\hline    &  &  &  &  & -26 & -6 & 32 &  & & \\
\hline   &  &  &  &  &   & -14 & -7  & 21 &  &\\
\hline   &  &  &   &  &  &  & -25 & -8 & 33 &   \\
\hline    &  &  &  &  &  &  &  & -13 & -9 & 22 \\
\hline  34 &  &  &  &  &   &  &   & & -24 & -10\\
\hline
\end{array}$
\end{footnotesize}
\end{center}
\end{ex}

\begin{prop}\label{prop:3even}
There exists an integer cyclically $3$-diagonal  $\H_3(n;3)$ for $n\equiv 0 \pmod 4$.
\end{prop}

\begin{proof}
We construct an $n\times n$ array $A$ using the following procedures labeled $\texttt{A}$ to $\texttt{J}$:
$$\begin{array}{lcl}
\texttt{A}:\; diag\left(2,2,1,1,1,\frac{n-4}{2}\right); & \hfill &
\texttt{B}:\; diag\left(\frac{n+6}{2},\frac{n+6}{2},-\frac{n+4}{2},1, -1,\frac{n-4}{2}\right);\\[3pt]
\texttt{C}:\; diag\left(2,1,-\frac{5n+4}{2},2,-1,\frac{n}{4}\right); & &
\texttt{D}:\; diag\left(3,2,-\frac{3n+2}{2},2,-1,\frac{n-4}{4}\right);\\[3pt]
\texttt{E}:\; diag\left(1,2,\frac{3n}{2},2,-1,\frac{n}{4}\right); &&
\texttt{F}:\; diag\left(2,3,\frac{5n+2}{2},2,-1,\frac{n-4}{4}\right);\\[3pt]
\texttt{G}:\; diag\left(\frac{n+6}{2},\frac{n+4}{2},-\frac{5n}{4},2,1,\frac{n}{4}\right); &&
\texttt{H}:\; diag\left(\frac{n+8}{2},\frac{n+6}{2},-\frac{9n}{4},2,1,\frac{n-4}{4}\right);\\[3pt]
\texttt{I}:\; diag\left(\frac{n+4}{2},\frac{n+6}{2},\frac{11n+8}{4},2,1,\frac{n}{4}\right); &&
\texttt{J}:\; diag\left(\frac{n+6}{2},\frac{n+8}{2},\frac{7n+8}{4},2,1,\frac{n-4}{4}\right).
\end{array}$$
We also fill the following cells in an \textit{ad hoc} manner:
$$\begin{array}{lclcl}
A\left[1,1\right]=-\frac{n-2}{2}, & \quad  &
A\left[\frac{n}{2},\frac{n}{2}\right]=n, & \quad &
A\left[\frac{n}{2},\frac{n+2}{2}\right]=\frac{7n+4}{4}, \\[3pt]
A\left[\frac{n+2}{2},\frac{n}{2}\right]=-\frac{9n+4}{4},& &
A\left[\frac{n+2}{2},\frac{n+2}{2}\right]=\frac{n+2}{2},& &
A\left[\frac{n+2}{2},\frac{n+4}{2}\right]=\frac{7n}{4}, \\[3pt]
A\left[\frac{n+4}{2},\frac{n+2}{2}\right]=-\frac{9n+8}{4}, &&
A\left[\frac{n+4}{2},\frac{n+4}{2}\right]=-\frac{n}{2}.
  \end{array}$$

We now prove that the array constructed above is an integer  $\H_3(n;3)$.
To aid in the proof we give a schematic picture of where each of the diagonal procedures fills cells (see Figure \ref{fig3}).
We have placed an $\texttt{X}$ in the \textit{ad hoc} cells.
Note that each row and each column contains exactly $3$ elements. We now check that the elements in every row sum to $0$
(in $\Z$).

\begin{figure}
\begin{footnotesize}
\begin{center}
\begin{tabular}{|c|c|c|c|c|c|c|c|c|c|c|c|}
  \hline
 \texttt{X} & \texttt{E} &  & & & & & & & &  & \texttt{G}\\ \hline
 \texttt{C} & \texttt{A} & \texttt{F} & & &  & & & & & &  \\ \hline
 &\texttt{D} & \texttt{A} & \texttt{E} & & & & & &  & &   \\ \hline
 & &\texttt{C} & \texttt{A} & \texttt{F}  & & & & & & &    \\ \hline
 & & &\texttt{D} & \texttt{A} & \texttt{E} & & & & & &     \\ \hline
  & & & &\texttt{C} & \texttt{X} & \texttt{X} & & & & &    \\ \hline
 & & & & &\texttt{X} & \texttt{X} & \texttt{X} & & & &   \\ \hline
 & & & & & &\texttt{X} & \texttt{X} & \texttt{I} & & &    \\ \hline
 & & & & & & &\texttt{G} & \texttt{B} & \texttt{J} & &    \\ \hline
 & & & & & & & &\texttt{H} & \texttt{B} & \texttt{I} &    \\ \hline
& & & & & & & & &\texttt{G} & \texttt{B} & \texttt{J}    \\ \hline
\texttt{I} & & & & & & & & & & \texttt{H} & \texttt{B}     \\ \hline
\end{tabular}
\end{center}
\end{footnotesize}
\caption{Scheme of construction with $n=12$.}\label{fig3}
\end{figure}

\begin{description}
\item[Row $1$] There is an \textit{ad hoc} value plus the first of the \texttt{E} diagonal as well as the last of the
\texttt{G} diagonal.
The sum is $$-\frac{n-2}{2}+\frac{3n}{2}-(n+1)=0.$$
\item[Row $2$  to $\frac{n-2}{2}$] There are two cases depending on whether the row $r$ is even or odd.
If $r$ is even, then write $r=2i+2$ where $i\in\left[0,\frac{n}{4}-2\right]$. Notice that from the \texttt{C},
\texttt{A} and \texttt{F} diagonal cells we get the following sum:
$$\left(-\frac{5n+4}{2}-i\right)+\left(1+2i\right)+\left(\frac{5n+2}{2}-i\right)=0.$$
If $r$ is odd, then write $r=2i+3$ where $i\in\left[0,\frac{n}{4}-2\right]$.
From the \texttt{D}, \texttt{A} and \texttt{E} diagonal cells we get the following sum:
$$\left(-\frac{3n+2}{2}-i\right)+\left(2+2i\right)+\left(\frac{3n-2}{2}-i\right)=0.$$
\item[Row $\frac{n}{2}$] We add the last value of the \texttt{C} diagonal with two  \textit{ad hoc} values:
$$-\frac{11n+4}{4}+n+\frac{7n+4}{4}=0.$$
\item[Row $\frac{n+2}{2}$] This row contains $3$ \textit{ad hoc} values. The sum is
$$-\frac{9n+4}{4}+\frac{n+2}{2}+\frac{7n}{4}=0.$$
\item[Row $\frac{n+4}{2}$] There are two \textit{ad hoc} values plus the first of the \texttt{I} diagonal.
The sum is $$-\frac{9n+8}{4}-\frac{n}{2}+\frac{11n+8}{4}=0.$$
\item[Row $\frac{n+6}{2}$ to $n$] Note that $\frac{n+6}{2}$ is odd. There are two cases depending on whether the row $r$ is odd or even.
If $r$ is odd, then write $r=\frac{n+6}{2}+2i$ where $i\in\left[0,\frac{n}{4}-2\right]$.
Notice that from the \texttt{G}, \texttt{B} and \texttt{J} diagonal cells we get the following sum:
$$\left(-\frac{5n}{4}+i\right)+\left(-\frac{n+4}{2}-2i\right)+\left(\frac{7n+8}{4}+i\right)=0.$$
If $r$ is even, then write $r=\frac{n+8}{2}+2i$ where $i\in\left[0,\frac{n}{4}-2\right]$.
Notice that from the \texttt{H}, \texttt{B} and \texttt{I} diagonal cells we get the following sum:
$$\left(-\frac{9n}{4}+i\right)+\left(-\frac{n+6}{2}-2i\right)+\left(\frac{11n+12}{4}+i\right)=0.$$
\end{description}
So we have shown that all row sums are zero. Next we check that the columns all add to zero.
\begin{description}
\item[Column $1$] There is an \textit{ad hoc} value plus the first of the \texttt{C} diagonal as well as the last of the
\texttt{I} diagonal.
The sum is $$-\frac{n-2}{2}-\frac{5n+4}{2}+(3n+1)=0.$$
\item[Column $2$  to $\frac{n-2}{2}$] There are two cases depending on whether the column $c$ is even or odd.
If $c$ is even, then write $c=2i+2$ where $i\in\left[0,\frac{n}{4}-2\right]$.
Notice that from the \texttt{E}, \texttt{A} and \texttt{D} diagonal cells we get the following sum:
$$\left(\frac{3n}{2}-i\right)+(1+2i)+\left(-\frac{3n+2}{2}-i\right)=0.$$
If $c$ is odd, then write $c=2i+3$ where $i\in\left[0,\frac{n}{4}-2\right]$.
From the \texttt{F}, \texttt{A} and \texttt{C} diagonal cells we get the following sum:
$$\left(\frac{5n+2}{2}-i\right)+(2+2i)+\left(-\frac{5n+6}{2}-i\right)=0.$$
\item[Column $\frac{n}{2}$] We add the last value of the \texttt{E} diagonal with two  \textit{ad hoc} values:
$$\frac{5n+4}{4}+n-\frac{9n+4}{4}=0.$$
\item[Column $\frac{n+2}{2}$] This column contains $3$ \textit{ad hoc} values. The sum is
$$\frac{7n+4}{4}+\frac{n+2}{2}-\frac{9n+8}{4}=0.$$
\item[Column $\frac{n+4}{2}$] There are two \textit{ad hoc} values plus the first of the \texttt{G} diagonal.
The sum is $$\frac{7n}{4}-\frac{n}{2}-\frac{5n}{4}=0.$$
\item[Column $\frac{n+6}{2}$ to $n$]  There are two cases depending on whether the column $c$ is odd or even.
If $c$ is odd, then write $c=\frac{n+6}{2}+2i$ where $i\in\left[0,\frac{n}{4}-2\right]$.
Notice that from the \texttt{I}, \texttt{B} and \texttt{H} diagonal cells we get the following sum:
$$\left(\frac{11n+8}{4}+i\right)+\left(-\frac{n+4}{2}-2i\right)+\left(-\frac{9n}{4}+i\right)=0.$$
If $c$ is even, then write $c=\frac{n+8}{2}+2i$ where $i\in\left[0,\frac{n}{4}-2\right]$.
Notice that from the \texttt{J}, \texttt{B} and \texttt{G} diagonal cells we get the following sum:
$$\left(\frac{7n+8}{4}+i\right)+\left(-\frac{n+6}{2}-2i\right)+\left(-\frac{5n-4}{4}+i\right)=0.$$
\end{description}
So we have shown that each column sums to $0$. Now we consider the support of $A$:
$$
\begin{array}{rcl}
supp(A) & =& \left[1,\frac{n-4}{2}\right]_{(\texttt{A})}\cup \left\{\frac{n-2}{2}, \frac{n}{2},\frac{n+2}{2}\right\} \cup
\left[\frac{n+4}{2},n-1\right]_{(\texttt{B})}\cup \{ n\} \cup\\[3pt]
&& \left[n+1,\frac{5n}{4}\right]_{(\texttt{G})} \cup  \left[\frac{5n+4}{4},\frac{3n}{2}\right]_{(\texttt{E})}
\cup\left[\frac{3n+2}{2},\frac{7n-4}{4}  \right]_{(\texttt{D})}\cup
\left\{\frac{7n}{4}, \frac{7n+4}{4}\right\}\cup\\[3pt]
&& \left[\frac{7n+8}{4},2n\right]_{(\texttt{J})} \cup
\left[2n+2,\frac{9n}{4}\right]_{(\texttt{H})}\cup  \{ \frac{9n+4}{4}, \frac{9n+8}{4}\} \cup
\left[\frac{9n+12}{4},\frac{5n+2}{2}\right]_{(\texttt{F})} \\[3pt]
&& \cup\left[\frac{5n+4}{2},\frac{11n+4}{4}\right]_{(\texttt{C})}\cup
\left[\frac{11n+8}{4},3n+1\right]_{(\texttt{I})}\\
&= & [1,3n+1]\setminus \{2n+1\}.
\end{array}$$
Thus, $A$ is an integer cyclically $3$-diagonal $\H_3(n;3)$ for $n\equiv 0\pmod 4$.
\end{proof}

The integer $\H_3(12;3)$ obtained following the proof of Proposition \ref{prop:3even}
is given in Example \ref{123}.

\begin{prop}\label{prop:5}
Let $n\equiv 3 \pmod 4$ with $n\geq7$. Then  there exists an integer cyclically $5$-diagonal $\H_5(n;5)$.
\end{prop}
\begin{proof}
We construct an $n\times n$ array $A$ using the following procedures
 labeled $\texttt{A}$ to  $\texttt{N}$.
$$\begin{array}{rlcrl}
\texttt{A}: &  diag\left(3,3,\frac{n-3}{2},2,-1,\frac{n-5}{2}\right); & \hfill &
\texttt{B}: &  diag\left(4,4,-(n-2),2,1,\frac{n-3}{2}\right);\\[3pt]
\texttt{C}: &  diag\left(3,2,2n+2,2,2,\frac{n-1}{2}\right); &&
\texttt{D}: &  diag\left(4,3,2n-1,2,-2,\frac{n-3}{2}\right);\\[3pt]
\texttt{E}: &  diag\left(2,3,-2n,2,2,\frac{n-1}{2}\right); &  &
\texttt{F}: &  diag\left(3,4,-(2n+3),2,-2,\frac{n-3}{2}\right);\\[3pt]
\texttt{G}: &  diag\left(3,1,-\frac{15n+7}{4},4,1,\frac{n-3}{4}\right); &&
\texttt{H}: &  diag\left(4,2,-(3n+4),4,-1,\frac{n+1}{4}\right);\\[3pt]
\texttt{I}: &  diag\left(5,3,-\frac{19n-1}{4},4,1,\frac{n-3}{4}\right);&&
\texttt{J}: &  diag\left(6,4,-(4n+3),4,-1,\frac{n-3}{4}\right);\\[3pt]
\texttt{K}: &  diag\left(1,3,\frac{17n+9}{4},4,1,\frac{n-3}{4}\right); &&
\texttt{L}: &  diag\left(2,4,5n,4,-1,\frac{n+1}{4}\right);\\[3pt]
\texttt{M}: &  diag\left(3,5,\frac{13n+17}{4},4,1,\frac{n-3}{4}\right); &&
\texttt{N}: &  diag\left(4,6,4n+1,4,-1,\frac{n-3}{4}\right).
\end{array}$$
We also fill the following cells in an \textit{ad hoc} manner:
$$\begin{array}{lclcl}
A\left[1,1\right]=n, & \hfill &  A\left[1,2\right]=-3n, &  \hfill & A\left[1,n\right]=n+1,\\[3pt]
A\left[2,1\right]=n+2, & & A\left[2,2\right]=n-1, &  & A\left[2,n\right]=-(5n+1),\\[3pt]
A\left[n-2,n-2\right]=-\frac{n-1}{2}, & & A\left[n-2,n\right]=5n+2, && A\left[n,1\right]=-(3n+1),\\[3pt]
A\left[n,2\right]=3n+3, & & A\left[n,n-2\right]=-(3n+2), & & A\left[n,n\right]=1.
\end{array}$$

We now prove that the array constructed above is an integer $\H_5(n;5)$.
To aid in the proof we give a schematic picture of where each of the diagonal procedures fills cells (see Figure \ref{fig4}).
We have placed an $\texttt{X}$ in the \textit{ad hoc} cells.
Note that each row and each column contains exactly $5$ elements. We now check that the elements in every row sum to $0$
(in $\Z$).

\begin{figure}
 \begin{footnotesize}
\begin{center}
\begin{tabular}{|c|c|c|c|c|c|c|c|c|c|c|c|c|c|c|}
  \hline
 \texttt{X} & \texttt{X} & \texttt{K} & & & & & & & & & & & \texttt{H}&  \texttt{X}\\ \hline
 \texttt{X} & \texttt{X} & \texttt{E} &\texttt{L} & & & & & & & & & & &\texttt{X}  \\ \hline
\texttt{G} &\texttt{C} & \texttt{A} & \texttt{F} &\texttt{M} & & & & & & & & & &   \\ \hline
 & \texttt{H}&\texttt{D} & \texttt{B} & \texttt{E}  &\texttt{N} & & & & & & & & &   \\ \hline
 & &\texttt{I} &\texttt{C} & \texttt{A} & \texttt{F} &\texttt{K} & & & & & & & &    \\ \hline
& & & \texttt{J}&\texttt{D} & \texttt{B} & \texttt{E}  &\texttt{L} & & & & & & &   \\ \hline
& & & &\texttt{G} &\texttt{C} & \texttt{A} & \texttt{F} &\texttt{M} & & & & & &   \\ \hline
& & & & &\texttt{H} &\texttt{D} & \texttt{B} & \texttt{E} &\texttt{N} & & & & &   \\ \hline
& & & & & & \texttt{I}&\texttt{C} & \texttt{A} & \texttt{F} &\texttt{K} & & & &    \\ \hline
& & & & & & & \texttt{J}&\texttt{D} & \texttt{B} & \texttt{E} &\texttt{L} & & &    \\ \hline
& & & & & & & &\texttt{G} &\texttt{C} & \texttt{A} & \texttt{F} &\texttt{M} & &     \\ \hline
& & & & & & & & & \texttt{H}&\texttt{D} & \texttt{B} & \texttt{E} &\texttt{N} &     \\ \hline
& & & & & & & & & &\texttt{I} &\texttt{C} & \texttt{X} & \texttt{F} & \texttt{X}    \\ \hline
\texttt{L}& & & & & & & & & & & \texttt{J}&\texttt{D} & \texttt{B} & \texttt{E}    \\ \hline
\texttt{X} &\texttt{X} & & & & & & & & & & &\texttt{X} & \texttt{C} & \texttt{X}     \\ \hline
\end{tabular}
\end{center}
\end{footnotesize}
\caption{Scheme of construction with $n=15$.}\label{fig4}
\end{figure}

\begin{description}
\item[Row $1$] We add  three \textit{ad hoc} values and the first  of the \texttt{K} diagonal with the last  of the
\texttt{H} diagonal:  $$n-3n+\frac{17n+9}{4}-\frac{13n+13}{4}+(n+1)=0.$$
\item[Row $2$] There are three \textit{ad hoc} values plus the first of the \texttt{E} diagonal as well as the first of
the \texttt{L} diagonal. The sum is $$(n+2)+(n-1)-2n+5n-(5n+1)=0.$$
\item[Row $3$ to $n-3$] Consider the row $r$, there are four cases according to the congruence class of $r$ modulo $4$.
If $r\equiv 3 \pmod 4$ write $r=3+4i$ where $i\in \left[0, \frac{n-7}{4}\right]$. Notice that from the \texttt{G},
\texttt{C}, \texttt{A}, \texttt{F} and \texttt{M} diagonal cells we get the following sum:
$$\left (-\frac{15n+7}{4}+i\right)+(2n+2 +4i)+\left
(\frac{n-3}{2}-2i\right)-(2n+3+4i)+\left(\frac{13n+17}{4}+i\right)=0.$$
If $r\equiv 0 \pmod 4$ write $r=4+4i$ where $i\in \left[0, \frac{n-7}{4}\right]$. Notice that from the \texttt{H},
\texttt{D}, \texttt{B}, \texttt{E} and \texttt{N} diagonal cells we get the following sum:
$$-(3n+4+i)+(2n-1 -4i)-(n-2-2i)-(2n-2-4i)+(4n+1-i)=0.$$
If $r\equiv 1 \pmod 4$ write $r=5+4i$ where $i\in \left[0, \frac{n-11}{4}\right]$. Notice that from the \texttt{I},
\texttt{C}, \texttt{A}, \texttt{F} and \texttt{K} diagonal cells we get the following sum:
$$\left(-\frac{19n-1}{4}+i\right)+(2n+4+4i)+\left(\frac{n-5}{2}-2i\right)-(2n+5+4i)+\left(\frac{17n+13}{4}+i\right)=0.$$
If $r\equiv 2 \pmod 4$ write $r=6+4i$ where $i\in \left[0, \frac{n-11}{4}\right]$. Notice that from the \texttt{J},
\texttt{D}, \texttt{B}, \texttt{E} and \texttt{L} diagonal cells we get the following sum:
$$-(4n+3+i)+(2n-3 -4i)-(n-3-2i)-(2n-4-4i)+(5n-1-i)=0.$$
\item[Row $n-2$] This row contains two \textit{ad hoc} values, the last of the \texttt{I} diagonal, the
$\frac{n-3}{2}$-th element of the \texttt{C} diagonal and the last of \texttt{F} diagonal. The sum is
    $$-\frac{9n+3}{2}+(3n-3)-\frac{n-1}{2}-(3n-2)+(5n+2)=0.$$
\item[Row $n-1$] We add the last elements of  the \texttt{L}, \texttt{J},  \texttt{D},  \texttt{B} and \texttt{E}
diagonals: $$\frac{19n+3}{4}-\frac{17n+5}{4}+(n+4)-\frac{n+1}{2}-(n+3)=0.$$
\item[Row $n$] This row contains four \textit{ad hoc} values and the last of \texttt{C} diagonal. The sum is
$$-(3n+1)+(3n+3)-(3n+2)+(3n-1)+1=0.$$
   \end{description}
So we have shown that all row sums are zero. Next we check that the columns all add to zero.
\begin{description}
\item[Column $1$] We add  three \textit{ad hoc} values and the first  of the \texttt{G} diagonal with the last  of the
\texttt{L} diagonal:
$$n+(n+2)-\frac{15n+7}{4}+\frac{19n+3}{4}-(3n+1)=0.$$
\item[Column $2$] There are three \textit{ad hoc} values plus the first of the \texttt{C} diagonal as well as the first
of the \texttt{H} diagonal. The sum is $$-3n+(n-1)+(2n+2)-(3n+4)+(3n+3)=0.$$
\item[Column $3$ to $n-3$] Consider the column $c$, there are four cases according to the congruence class of $c$ modulo
$4$.
If $c\equiv 3 \pmod 4$ write $c=3+4i$ where $i\in \left[0, \frac{n-7}{4}\right]$. Notice that from the \texttt{K},
\texttt{E}, \texttt{A}, \texttt{D} and \texttt{I} diagonal cells we get the following sum:
$$\left(\frac{17n+9}{4}+i\right)-(2n-4i)+\left (\frac{n-3}{2}-2i\right)+(2n-1-4i)-\left(\frac{19n-1}{4}-i\right)=0.$$
If $c\equiv 0 \pmod 4$ write $c=4+4i$ where $i\in \left[0, \frac{n-7}{4}\right]$. Notice that from the \texttt{L},
\texttt{F}, \texttt{B}, \texttt{C} and \texttt{J} diagonal cells we get the following sum:
$$(5n-i)-(2n+3+4i)-(n-2-2i)+(2n+4+4i)-(4n+3+i)=0.$$
If $c\equiv 1 \pmod 4$ write $c=5+4i$ where $i\in \left[0, \frac{n-11}{4}\right]$. Notice that from the \texttt{M},
\texttt{E}, \texttt{A}, \texttt{D} and \texttt{G} diagonal cells we get the following sum:
$$\left(\frac{13n+17}{4}+i\right)-(2n-2-4i)+\left(\frac{n-5}{2}-2i\right)+(2n-3-4i)-\left(\frac{15n+3}{4}-i\right)=0.$$
If $c\equiv 2 \pmod 4$ write $c=6+4i$ where $i\in \left[0, \frac{n-11}{4}\right]$. Notice that from the \texttt{N},
\texttt{F}, \texttt{B}, \texttt{C} and \texttt{H} diagonal cells we get the following sum:
$$(4n+1-i)-(2n+5+4i)-(n-3-2i)+(2n+6+4i)-(3n+5+i)=0.$$
\item[Column $n-2$] This column contains two \textit{ad hoc} values, the last of the \texttt{M} diagonal, the
$\frac{n-3}{2}$-th element of the \texttt{E} diagonal and the last of \texttt{D} diagonal. The sum is
    $$\frac{7n+5}{2}-(n+5)-\frac{n-1}{2}+(n+4)-(3n+2)=0.$$
\item[Column $n-1$] We add the last elements of  the \texttt{H}, \texttt{N},  \texttt{F},  \texttt{B} and \texttt{C}
diagonals:
$$-\frac{13n+13}{4}+\frac{15n+11}{4}-(3n-2)-\frac{n+1}{2}+(3n-1)=0.$$
\item[Column $n$] This column contains four \textit{ad hoc} values and the last of \texttt{E} diagonal. The sum is
$$(n+1)-(5n+1)+(5n+2)-(n+3)+1=0.$$
\end{description}
So we have shown that each column sums to $0$. Now we consider the support of $A$:
$$\begin{array}{rcl}
supp(A) & = & \{1\}\cup [2,\frac{n-3}{2}]_{(\texttt{A})}\cup \{\frac{n-1}{2}\}\cup
[\frac{n+1}{2},n-2]_{(\texttt{B})}\cup \\[3pt]
&& \{ n-1, n, n+1, n+2\}  \cup [n+3,2n]_{(\texttt{E},\texttt{D})}\cup
[2n+2,3n-1]_{(\texttt{C},\texttt{F})}\cup \\[3pt]
&& \{3n, 3n+1,3n+2,3n+3\} \cup [3n+4,\frac{13n+13}{4}]_{(\texttt{H})} \cup\\[3pt]
&& [\frac{13n+17}{4},\frac{7n+5}{2}]_{(\texttt{M})}\cup
[\frac{7n+7}{2},\frac{15n+7}{4}]_{(\texttt{G})}\cup [\frac{15n+11}{4},4n+1]_{(\texttt{N})}\cup\\[3pt]
&& [4n+3,\frac{17n+5}{4}]_{(\texttt{J})}  \cup[\frac{17n+9}{4},\frac{9n+1}{2}]_{(\texttt{K})}\cup
[\frac{9n+3}{2},\frac{19n-1}{4}]_{(\texttt{I})}\cup\\[3pt]
&&[\frac{19n+3}{4},5n]_{(\texttt{L})}\cup \{5n+1,5n+2\}\\[3pt]
& =& [1,5n+2]\setminus \{2n+1,4n+2\}.
 \end{array}$$
Thus, $A$ is an integer cyclically $5$-diagonal $\H_5(n;5)$ for  $n\equiv 3 \pmod 4$.
\end{proof}

\begin{ex}
Following the proof of Proposition \ref{prop:5} we obtain the
integer $\H_5(15;$ $5)$ below.
\begin{center}
\begin{tiny}
$\begin{array}{|r|r|r|r|r|r|r|r|r|r|r|r|r|r|r|}
\hline 15 & -45 & 66  &  &  &  &  & & & &&&& -52 & 16  \\
\hline  17 & 14 & -30 & 75 &  &  &  &  & & &&&&& -76 \\
\hline  -58 & 32 & 6 & -33 & 53  &  &  & & &&&&& & \\
\hline   & -49  & 29 & -13 & -28 & 61 &  & &&&& & & &\\
\hline   &  & -71 & 34   & 5 & -35  & 67 & & & & & & & &\\
\hline   &  &  & -63  &  27 & -12 & -26 & 74 & & & & & & & \\
\hline    &  &  &  & -57 & 36 & 4 & -37 & 54 & & & & & & \\
\hline   &  &  &  &  & -50  & 25 & -11  & -24 & 60& & & & &\\
\hline   &  &  &   &  &  & -70 & 38 & 3 & -39 & 68 & & & & \\
\hline   &  &  &   &  &  & &-64 & 23 & -10 & -22 & 73  & & & \\
\hline   &  &  &   &  &  & & &-56 & 40 & 2 & -41 & 55 & &  \\
\hline   &  &  &   &  &  & & & &-51 & 21 & -9 & -20 & 59  & \\
\hline   &  &  &   &  &  & & & & &-69 & 42 & -7 & -43 & 77 \\
\hline   72  &  &   &  &  & & & & & & &-65 & 19 & -8 & -18 \\
\hline   -46& 48  &  &   &  &  & & & & & &  & -47 & 44 & 1 \\
\hline
\end{array}$
\end{tiny}
\end{center}
\end{ex}

\begin{prop}\label{prop:6}
There exists an integer cyclically $(2,3)$-diagonal  $\H_6(n;6)$ for every even $n\geq 6$.
\end{prop}

\begin{proof}
Let $U$, $V_5$ and $V_9$ be the three $2\times 6$ arrays defined in the proof of  Proposition \ref{prop:6Finta}.
Also here we have to distinguish three cases.

If $n=6m$, let $A$ be a cyclically $(2,3)$-diagonal  $n\times n$ p.f. array whose strips $S_i$ are:
$$S_i =\left\{\begin{array}{ll}
U\pm 12i & \textrm{ if }  i \in [0,m-1],\\
U\pm (1+12i) & \textrm{ if }  i \in [m,2m-1],\\
U\pm (2+12i) & \textrm{ if }  i \in [2m,3m-1].\\
 \end{array}\right.$$
We obtain
$$\begin{array}{rcl}
   supp(A) & = & \bigcup_{i=0}^{m-1}[1+12i,12+12i]\cup \bigcup_{i=m}^{2m-1}[2+12i,13+12i] \cup \\[3pt]
&& \bigcup_{i=2m}^{3m-1}[3+12i,14+12i]\\[3pt]
& =& [1,12m]\cup [12m+2,24m+1] \cup [24m+3,36m+2].
\end{array}$$

If $n=6m+2$, let $A$ be a cyclically $(2,3)$-diagonal $n\times n$ p.f. array whose strips $S_i$ are:
$$S_i =\left\{\begin{array}{ll}
U\pm 12i & \textrm{ if }  i \in [0,m-1],\\
V_5\pm 12i & \textrm{ if }  i=m,\\
U\pm (1+12i) & \textrm{ if }  i \in [m+1,2m-1],\\
V_9\pm (1+12i)  & \textrm{ if }  i=2m,\\
U\pm (2+12i) & \textrm{ if }  i \in [2m+1,3m].\\
 \end{array}\right.$$
It follows that
$$\begin{array}{rcl}
   supp(A) & = & \bigcup_{i=0}^{m-1}[1+12i,12+12i]\cup [12m+1,12m+4]\cup \\[3pt]
   && [12m+6,12m+13] \cup    \bigcup_{i=m+1}^{2m-1}[2+12i,13+12i] \cup \\[3pt]
&& [24m+2, 24m+9] \cup [24m+11,24m+14] \cup \\[3pt]
   && \bigcup_{i=2m+1}^{3m}[3+12i,14+12i]\\[3pt]
& =& [1,12m+4]\cup [12m+6,24m+9] \cup [24m+11,36m+14].
  \end{array}$$

If $n=6m+4$, let $A$ be a cyclically $(2,3)$-diagonal $n\times n$ p.f. array whose strips $S_i$ are:
$$S_i =\left\{\begin{array}{ll}
U\pm 12i & \textrm{ if }  i \in [0,m-1],\\
V_9\pm 12i & \textrm{ if }  i=m,\\
U\pm (1+12i) & \textrm{ if }  i \in [m+1,2m],\\
V_5\pm (1+12i)  & \textrm{ if }  i=2m+1,\\
U\pm (2+12i) & \textrm{ if }  i \in [2m+2,3m+1].\\
 \end{array}\right.$$
It follows that
$$\begin{array}{rcl}
   supp(A) & = & \bigcup_{i=0}^{m-1}[1+12i,12+12i]\cup [12m+1,12m+8]\cup \\[3pt]
   && [12m+10,12m+13] \cup    \bigcup_{i=m+1}^{2m}[2+12i,13+12i] \cup \\[3pt]
&& [24m+14, 24m+17] \cup [24m+19,24m+26] \cup \\[3pt]
   && \bigcup_{i=2m+2}^{3m+1}[3+12i,14+12i]\\[3pt]
& =& [1,12m+8]\cup [12m+10,24m+17] \cup [24m+19,36m+26].
  \end{array}$$

In all three cases, we have
$$supp(A)=[1,6n+2]\setminus\{2n+1, 4n+2\}$$
and so the associated p.f. array $A$ we constructed is an integer $\H_6(n;6)$.
\end{proof}

\begin{ex}
Following the proof of Proposition \ref{prop:6} we obtain the
integer $\H_6(10;$ $6)$ below.
  \begin{center}
\begin{footnotesize}
$\begin{array}{|r|r|r|r|r|r|r|r|r|r|}\hline
-1 & 5 & 2 & -7 & -9 & 10 &  &  &  & \\\hline
3 & -4 & -6 & 8 & 11 & -12 &  &  &  & \\\hline
 &  & -13 & 17 & 14 & -19 & 25 & -24 &  &\\ \hline
 &  & 15 & -16 & -18 & 20 & -23 & 22 &  & \\\hline
 &  &  &  & -26 & 30 & 27 & -32 & -34 & 35\\\hline
 &  &  &  & 28 & -29 & -31 & 33 & 36 & -37\\\hline
41 & -45 &  &  &  &  & -38 & 47 & 44 & -49\\\hline
-39 & 43 &  &  &  &  & 40 & -46 & -48 & 50\\\hline
52 & -57 & -59 & 60 &  &  &  &  & -51 & 55\\\hline
-56 & 58 & 61 & -62 &  &  &  &  & 53 & -54\\\hline
\end{array}$
\end{footnotesize}
\end{center}
\end{ex}

\section{Conclusions}\label{sec:conclusion}

Now we are ready to prove our main result.

\begin{proof}[Proof of Theorem \ref{thm:esistenza}]
We split the proof into 4 cases according to the congruence class of  $k$ modulo $4$.

Case 1. Let $k\equiv 3 \pmod 4$. An integer cyclically $3$-diagonal $\H_3(n;3)$ for $n\equiv 3 \pmod 4$
and $n\equiv 0 \pmod 4$ has been constructed in Proposition \ref{prop:3odd} and in Proposition \ref{prop:3even}, respectively.
The result follows applying inductively Theorem \ref{prop:ExtDiag}(1).

Case 2.  Let $k\equiv 0 \pmod 4$. An integer cyclically $4$-diagonal $\H_4(n;4)$ for any $n\geq4$
has been constructed in Proposition \ref{prop:4}.
As before, the result follows applying inductively Theorem \ref{prop:ExtDiag}(1).

Case 3. Let $k\equiv 1 \pmod 4$.
If $n\equiv 3 \pmod 4$ an integer cyclically $5$-diagonal $\H_5(n;5)$
has been constructed in Proposition \ref{prop:5}.
As before, the result follows applying inductively Theorem \ref{prop:ExtDiag}(1).
If $n\equiv 0\pmod 4$ and $n > k\geq9$, by Case 1, there exists an integer cyclically $(k-6)$-diagonal $\H_{k-6}(n;k-6)$,
which can be viewed also as an integer cyclically $(2,\frac{k-5}{2})$-diagonal  $\H_{k-6}(n;k-6)$ by Remark \ref{cyc}.
Since $n\equiv 0\pmod 4$ and $n\geq k+1$, the result follows applying Theorem \ref{prop:ExtDiag}(3).

Case 4. Let $k\equiv 2 \pmod 4$. An integer cyclically $(2,3)$-diagonal $\H_6(n;6)$ for $n$ even with $n\geq 6$
 has been constructed in Proposition \ref{prop:6}.
The result follows applying inductively Theorem \ref{prop:ExtDiag}(2).
\end{proof}

As already remarked in Section \ref{sec:Intro} we leave open the existence problem of an integer $\H_k(n;k)$
only for $k=5$ and $n\equiv0\pmod4$. For this class we have two examples:
an integer cyclically $(2,3)$-diagonal $\H_5(8;5)$ is given in Example \ref{ex:165}, while
an integer cyclically $(2,3)$-diagonal $\H_5(16;5)$
is given below.

 \begin{center}
\begin{Tiny}
$\begin{array}{|r|r|r|r|r|r|r|r|r|r|r|r|r|r|r|r|}
\hline  8&& -65& 81&&&&&&&&&&& 55& -79 \\
\hline& 16& 82& -58&&&&&&&&&&& -80& 40 \\
\hline 60& -77& -6&& -38& 61&&&&&&&&&& \\
\hline -78& 53&& -14& 62& -23&&&&&&&&&& \\
\hline&& 21& -31& -5&& 57& -42&&&&&&&& \\
\hline&& -32& 22&& -13& -41& 64&&&&&&&& \\
\hline&&&& -69& 51& 7&& -17& 28&&&&&& \\
\hline&&&& 50& -76&& 15& 29& -18&&&&&& \\
\hline&&&&&& 45& -67& 3&& 39& -20&&&& \\
\hline&&&&&& -68& 30&& 11& -19& 46&&&& \\
\hline&&&&&&&& 56& -70& 2&& -24& 36&& \\
\hline&&&&&&&& -71& 49&& 10& 37& -25&& \\
\hline&&&&&&&&&& -48& 27& 1&& -54& 74 \\
\hline&&&&&&&&&& 26& -63&& 9& 75& -47 \\
\hline -34& 43&&&&&&&&&&& 59& -72& 4& \\
\hline 44& -35&&&&&&&&&&& -73& 52&& 12 \\
\hline
\end{array}$
\end{Tiny}
\end{center}

Our existence result about relative Heffter arrays implies the existence of new pairs of
orthogonal cycle decompositions. To describe how this result is obtained, we first recall
the following conjecture.

\begin{conj}\cite[Conjecture 3]{CMPPSums}\label{Co}
Let $(G,+)$ be an abelian group. Let $A$ be a finite subset of $G\setminus\{0\}$ such that no $2$-subset $\{x,-x\}$ is contained in $A$
and with the property that $\sum_{a\in A} a=0$. Then there exists a simple ordering of the elements of $A$.
\end{conj}

\begin{prop}
For $n\geq k$,
there exist two orthogonal cyclic $k$-cycle decompositions of $K_{(2n+1)\times k}$ in each of the following cases:
\begin{itemize}
  \item[(1)] $k=3,7,9$ for  $n\equiv0,3\pmod4$;
  \item[(2)] $k=5$ for  $n\equiv 3 \pmod 4$;
  \item[(3)] $k=4,8$ for every $n$;
  \item[(4)] $k=6$ for every even $n$.
\end{itemize}
\end{prop}

\begin{proof}
In \cite{CMPPSums}, we verified Conjecture \ref{Co} for any set of size less than $10$.
Hence for $k\leq 9$ all the $\H_k(n;k)$ here constructed are simple for any $n$.
The result follows from Theorem \ref{thm:esistenza} and Proposition \ref{HeffterToDecompositions}.
\end{proof}

If Conjecture \ref{Co} were true we would have two orthogonal cyclic $k$-cycle decompositions
of $K_{(2n+1)\times k}$ for any pair $(n,k)$
for which we have constructed an integer $\H_k(n;k)$.

\section*{Acknowledgements}

This research was partially supported by Italian Ministry of Education, Universities and Research under Grant
PRIN 2015 D72F16000790001 and by INdAM-GNSAGA.

\end{document}